\documentclass[11pt,reqno]{amsart}

\usepackage{lineno,hyperref}
\modulolinenumbers[5]
\usepackage{amsmath}
\usepackage{amssymb}
\usepackage{mathrsfs}
\usepackage{amsthm}
\usepackage{bm}
\usepackage{graphicx}
\usepackage{booktabs}
\usepackage{float}
\usepackage{subfig}
\usepackage{geometry}
\newtheorem{Def}{Definition}[section]
\newtheorem{lem}[Def]{Lemma}
\newtheorem{tho}[Def]{Theorem}

\newtheorem{rem}[Def]{Remark}
\newcommand{\ud}{\mathrm d}
\newcommand{\dw}{\Delta \widehat W}
\newcommand{\dtau}{\mathrm d\tau}
\newcommand{\dsit}{\mathrm d\theta}

\numberwithin{equation}{section}
\allowdisplaybreaks[4]

\begin{document}

%\begin{frontmatter}

\title[Modified averaged vector field methods]{Modified averaged vector field methods preserving multiple invariants for conservative stochastic differential equations}
%multi-invariant-preserving numerical methods for conservative stochastic differential equations

%\subjclass[2010]{Primary 60H35;60H15}
\author{Chuchu Chen}%\corref{A1}}
\address{Academy of Mathematics and Systems Science, Chinese Academy of Sciences, Beijing
	100190, China; School of Mathematical Sciences, University of Chinese Academy of
	Sciences, Beijing 100049, China}
%\fntext[A1]{Academy of Mathematics and Systems Science, Chinese Academy of Sciences, Beijing, China}
\email{chenchuchu@lsec.cc.ac.cn}

\author{Jialin Hong}%\corref{A1}}
\address{Academy of Mathematics and Systems Science, Chinese Academy of Sciences, Beijing
	100190, China; School of Mathematical Sciences, University of Chinese Academy of
	Sciences, Beijing 100049, China}
%\fntext[A2]{Academy of Mathematics and Systems Science, Chinese Academy of Sciences, Beijing, China}
\email{hjl@lsec.cc.ac.cn}

\author{Diancong Jin}%\corref{A1}}
\address{Academy of Mathematics and Systems Science, Chinese Academy of Sciences, Beijing
	100190, China; School of Mathematical Sciences, University of Chinese Academy of
	Sciences, Beijing 100049, China}
\email{diancongjin@lsec.cc.ac.cn (Corresponding author)}
%\cortext[A1]{Academy of Mathematics and Systems Science, Chinese Academy of Sciences, Beijing
%100190, China; School of Mathematical Sciences, University of Chinese Academy of
%Sciences, Beijing 100049, China

%\ead{diancongjin@lsec.cc.ac.cn (Corresponding author)}
\thanks{This work is supported by National Natural Science Foundation of China (NO. 91530118, NO. 91130003, NO. 11021101, NO. 91630312 and NO. 11290142).}


\begin{abstract}
%A novel class of modified averaged vector field (MAVF) methods  for general conservative Stratonovich stochastic differential equations with multiple noises is constructed and analyzed, in order to preserve exactly multiple invariants of original stochastic system. The mean square convergence order $1$ of MAVF methods is proved in the case of commutative noises. Numerical experiments are performed to verify the theoretical analyses and to present the superiority of the proposed methods in long time simulation.
A novel class of conservative numerical methods for general conservative Stratonovich stochastic differential equations with multiple invariants is proposed and analyzed.  These methods, which are called modified averaged vector field methods, are constructed by modifying the averaged vector field  methods to preserve multiple invariants simultaneously. Based on the prior estimate for high order moments of the modification coefficient, the mean square convergence order $1$ of proposed methods is proved in the case of commutative noises. In addition, the effect  of quadrature formula on the mean square convergence order and the preservation of invariants for the modified averaged vector field  methods is considered.
%It is proved that  the induced modified averaged vector field methods are still of mean square order $1$ provided that the orders of quadrature formula are no less than $2$. Generally,  the mean square order of invariants conservation of modified averaged vector field methods using numerical integration only depends on the order of quadrature formulas.  
Numerical experiments are performed to verify the theoretical analyses and to show the superiority of the proposed methods in long time simulation.
\end{abstract}

%\end{frontmatter}
%\linenumbers
\maketitle

\section{Introduction}
Numerical methods for stochastic differential equations (SDEs) have attracted extensive attention over the past decades, in view of the difficulty of obtaining the explicit solutions of original systems (see e.g. \cite{Burrage,Kloeden,MilsteinBook}). 
%In the numerical treatment of SDEs, 
%the  numerical methods preserving  invariants for SDEs are very important and have good long time behavior as the deterministic cases \cite{Geo}.
It is important to construct numerical methods which preserve properties  for original systems as much as possible. For conservative SDEs with one invariant, there have been many works related to numerical methods in recent years. On the one hand, aiming at the SDEs with single noise, \cite{Misadif} proposes an energy-preserving difference method for stochastic Hamiltonian systems and analyzes the local errors.  Based on the equivalent “skew gradient” (SG) form for conservative SDEs with one invariant, \cite{Zjj} proposes direct discrete gradient methods and indirect discrete gradient methods, and proves that these two kinds of methods are of mean square order $1$. Authors in \cite{Cohen} construct energy-preserving methods for stochastic Poisson systems, and prove that those methods are of mean square order one and preserve quadratic Casimir functions. On the other hand, in the case of SDEs with  multiple noises, \cite{Ccc} proposes the averaged vector field (AVF) methods for conservative SDEs. It is shown that the mean square order of AVF method is $1$ if noises are commutative and that the weak order is $1$ in the general case. For the case of quadratic invariants, \cite{Xds} constructs stochastic Runge-Kutta (SRK) methods for SDEs with quadratic invariants and \cite{Anne} gives the order conditions for SRK methods preserving quadratic invariants.
% As for the numerical methods preserving multiple invariants, projection methods, are proposed in \cite{Zwe} and prove to improve the mean square convergence order effectively.

For conservative SDEs with multiple invariants, one difficulty is to preserve multiple invariants simultaneously. One approach is via projection technique, which combines an arbitrary one-step approximation together with a projection onto the invariant submanifold in each step. \cite{Zwe} shows that this approach is feasible in stochastic settings, and the proposed methods could reach high strong order as supporting methods. In this paper, we focus on constructing a new class of multi-invariant-preserving methods, which are called modified averaged vector field (MAVF) methods. More precisely, we add modification terms to AVF methods to preserve  multiple invariants simultaneously, motivated by the ideas of line integral methods (LIMs) for  deterministic conservative ordinary differential equations (ODEs) in \cite{Brugbook}. 

%As is seen in \eqref{method1}, due to the introduction of modification coefficient $\bm{\alpha}=(\alpha_0,\alpha_1)$ (a vector-valued random variable), we have to estimate the high order moments of $\bm{\alpha}$ to acquire the convergence order of MAVF methods.
As is seen in \eqref{method1}, the modification terms contain a vector-valued random variable $\bm{\alpha}=(\alpha_0,\alpha_1)$ which is called modification coefficient, hence a prerequisite to acquire the convergence order of MAVF methods is the boundedness of the high-order moments of $\bm{\alpha}$.
To this end, one technique is to truncate the Brownian increments, which not only ensures the solvability of MAVF methods, but also makes sure that for sufficiently small stepsize, $\bm{\alpha}$ is uniformly small with respect to the sample path $\omega$ . Another technique is the usage of the orthogonality of Legendre polynomials, which makes us get rid of the effect of low-order terms and then acquire the estimate for high-order moments  of $\bm{\alpha}$. 
%Firstly, we introduce the truncated Brownian increments. On one hand, the introduction of Brownian increments is able to deal with the solvability of MAVF methods effectively. On other hand, it makes sure that $\bm{\alpha}$ is uniformly small with respect the sample path $\omega$ , for sufficiently small stepsize h. In addition, by using the orthogonality of the Legendre polynomials, we eliminate the low order terms and acquire the precise estimate of high moments of $\bm{\alpha}$. 
We compare MAVF methods with Milstein method to prove that MAVF methods are of mean square order $1$.

When the integrals contained in MAVF methods can not be obtained directly,   numerical integration is an option to approximate these integrals.  Thus it is necessary to investigate the effect of numerical integration   on the mean square convergence order and the preservation of invariants for the proposed methods. 
%We prove that the usage of numerical integration does not change the mean square order of MAVF methods, provided that the order of quadrature formula is no less than $2$. Furthermore, we study the relationship between order of quadrature formula and mean square order of invariant conservation for MAVF methods using numerical integration. 
It is proved that  the induced MAVF methods are still of mean square order $1$ provided that the orders of quadrature formula are no less than $2$. Generally,  the mean square order of invariants conservation of MAVF methods using numerical integration only depends on the order of quadrature formulas. 
%It is proved that the MAVF methods with numerical integration are ``almost conservative'', if the order of quadrature formula is enough large.

The rest of this paper is organized as follows. In Section \ref{S2}, we give some concepts about conservative SDEs with invariants and preliminary theorems and lemmas for numerical analyses. Section \ref{S3} proposes  MAVF methods for conservative SDEs with single or multiple noises and shows the properties of these methods.  Section \ref{S4} investigates the MAVF methods using numerical integration, and analyzes their convergence and preservation of invariant. Numerical experiments are performed, in Section \ref{S5}, to verify the theoretical analyses and to show the advantages of MAVF methods in long time simulations. 
%Finally, Section \ref{S5} gives the conclusion of the paper.

In the sequel, for convenience, we will use the following notations:
\begin{itemize}
  \item $|x|$: The trace norm of vector or matrix $x$ by $|x|=\sqrt{{\rm Tr}(x^\top x)}$.
%  \item $\nabla f$: The gradient of scalar function $f \in \mathbf C^1(\mathbb R^n, \mathbb R)$: $\nabla f=(\frac{\partial f}{\partial x_1},\ldots,\frac{\partial f}{\partial x_k})$ or Jacobian matrix of vector function $f$.
  \item $\mathbf C^k(\mathbb R^m, \mathbb R^n)$: The space of $k$ times continuously differentiable functions $f:\mathbb R^m\to \mathbb R^n$.
  \item $\mathbf C_b^k(\mathbb R^m, \mathbb R^n)$: The space of $k$ times continuously differentiable functions $f:\mathbb R^m\to \mathbb R^n$ with uniformly bounded derivatives up to order $k$. 
  \item $\nabla f$: The gradient of scalar function $f \in \mathbf C^1(\mathbb R^m, \mathbb R)$: $\nabla f=(\frac{\partial f}{\partial x_1},\ldots,\frac{\partial f}{\partial x_m})$ or Jacobian matrix of vector function $f \in\mathbf C^k(\mathbb R^m, \mathbb R^n)$: $\nabla f=(\nabla f_1^\top,\ldots,\nabla f_n^\top)^\top$.
  \item G\v ateaux derivative $f^{(k)}(x)(\xi_1,\ldots,\xi_k)$: If $f(x)\in\mathbf C^k(\mathbb R^n, \mathbb R)$ and $\xi_1,\ldots,\xi_k\,\in\mathbb R^n$, then $f^{(k)}(x)(\xi_1,\ldots,\xi_k)=\sum_{i_1,\ldots,i_k=1}^n\,\frac{\partial^k f(x)}{\partial x^{i_1}\cdots\partial x^{i_k}}\,\xi_1^{i_1}\cdots\xi_k^{i_k}$.
\end{itemize}
\section{Preliminary}\label{S2}
In this section, we give the definition of invariant for conservative SDEs and introduce some lemmas and theorems for the proof of convergence.

Consider the general $m$-dimensional autonomous SDE in the sense of Stratonovich
\begin{equation}\label{SDE1}
 \mathrm{d}Y(t)=f\big(Y(t)\big)\mathrm{d}t+\sum_{r=1}^D g_r\big(Y(t)\big)\circ\mathrm{d}W_r(t),\quad 0\leq t\leq T,\quad Y(0)=Y_0,
\end{equation}
where $W_r(t),\,r=1,\ldots,D$, are $D$ independent one-dimensional Brownian motions defined on a complete filtered probability space $\left(\Omega,\mathcal{F},\{\mathcal{F}_t\}_{t\geq0},P\right)$ with $\{\mathcal{F}_t\}_{t\geq0}$ satisfying the usual conditions. Assume that $Y_0$ is $\mathcal{F}_0$-measurable with $E|Y_0|^2<\infty$, and that $f:\mathbb{R}^m\to \mathbb{R}^m$, $g_r:\mathbb{R}^m\to \mathbb{R}^m,r=1,\ldots,D,$ are such that (\ref{SDE1}) has a unique global solution. Next we give the definition of invariant.
%\begin{Def} (see \cite{Zwe})
%A differentiable non-constant scalar function $I(y)$ is called an invariant (conservative quantity or first integral) of SDE (\ref{SDE1}) if
%\begin{equation*}
%  \nabla I(y)^\top f(y)=0,\quad\nabla I(y)^\top g_r(y)=0,\,r=1,\ldots,D,\quad\forall~y\in \mathbb{R}^m.
%\end{equation*}
%\end{Def}

\begin{Def}[see ~\rm{\cite{Zwe}}]
SDE (\ref{SDE1}) is said to have $\nu$ invariants $L^i(y)\in \mathbf C^1(\mathbb R^m, \mathbb R)$, $i=1,\ldots,\nu$, if 
\begin{equation}\label{2.2}
  \nabla L^i(y) f(y)=0,~\nabla L^i(y) g_r(y)=0,~r=1,\ldots,D,~i=1,\ldots,\nu,\quad \forall~ y\in\mathbb{R}^m.
\end{equation}
%Moreover, these invariants are called independent, if $\nabla L^i(y),\,i=1,\ldots,\nu$, are functionally independent.
\end{Def}

If we define vector-valued function $L(y):=\left(L^1(y),\ldots,L^\nu(y)\right)^\top$, then (\ref{2.2}) can be compactly written as 
\begin{equation}\label{2.3}
  \nabla L(y)f(y)=\nabla L(y)g_r(y)=\bm{0},~r=1,\ldots,D,\quad \forall ~y\in \mathbb{R}^m.
\end{equation}
Hereafter, we also say that the vector-valued function $L(y)$, which satisfies (\ref{2.3}), is the invariant of (\ref{SDE1}). According to the definition of invariants, it follows from stochastic chain rule that $\ud L(Y(t))=\bm{0}$, where $Y(t)$ is the exact solution of  (\ref{SDE1}). This implies that $L(Y(t))=L(Y_0)$, a.s. This is to say, $L(y)$, along the exact solution $Y(t)$, is invariant almost surely.
%It can be regarded as the extension of the deterministic case (see Chapter 4 in \cite{Geo}).

The following two theorems give the relationship between local errors and global errors of numerical methods for general SDEs.
% As is shown in \cite{MilsteinMS}, these theorems hold under conditions that the It\^o-type stochastic differential equation, converted from Stratonovich-type equation \eqref{SDE1},  has drift and diffusion terms which are global Lipschitz continuous.
In the sequel, we always assume that  the assumptions of these two theorems in \cite{MilsteinMS} hold unless we make additional statement.
\begin{tho}[see~\rm{\cite{MilsteinMS}}] \label{tho:MS}
Suppose the one-step approximation $\bar{X}_{t,x}(t+h)$ has order of accuracy $p_1$ for the expectation of the deviation and order of accuracy $p_2$ for the mean square deviation; more precisely, for arbitrary $t_0\leq t\leq t_0+T-h,\,x\in \mathbb{R}^d$ the following inequalities hold:
\begin{eqnarray}
  |E\left(X_{t,x}(t+h)-\bar{X}_{t,x}(t+h)\right)|\leq K\cdot\left(1+|x|^2\right)^{1/2}h^{p_1}, \\
  \left[E|X_{t,x}(t+h)-\bar{X}_{t,x}(t+h)|^2\right]^{1/2}\leq K\cdot\left(1+|x|^2\right)^{1/2}h^{p_2}.
\end{eqnarray}
Also, let
$$p_2\geq\frac{1}{2},~p_1\geq p_2+\frac{1}{2}.$$
Then for any $N$ and $k=0,\ldots,N$ the following inequality holds:
\begin{equation}
  \left[E|X_{t_0,X_0}(t_k)-\bar{X}_{t_0,X_0}(t_k)|^2\right]^{1/2}\leq K\cdot\left(1+E|X_0|^2\right)^{1/2}h^{p_2-1/2},
\end{equation}
i.e., the mean-square order of accuracy of the method constructed using the one-step approximation $\bar{X}_{t,x}(t+h)$ is $p=p_2-1/2$.
\end{tho}

\begin{tho} 
[see~ \rm{\cite{MilsteinMS}}]\label{tho:MSC}
Let the one-step approximation $\bar{X}_{t,x}(t+h)$ satisfy the conditions of Theorem \ref{tho:MS}. Suppose that $\tilde{X}_{t,x}(t+h)$ is such that
\begin{gather}
  |E\left(\tilde{X}_{t,x}(t+h)-\bar{X}_{t,x}(t+h)\right)|= \mathcal{O}(h^{p_1}), \\
  \left[E|\tilde{X}_{t,x}(t+h)-\bar{X}_{t,x}(t+h)|^2\right]^{1/2}= \mathcal{O}(h^{p_2}).
\end{gather}
with the same $p_1$ and $p_2$. Then the method based on the one-step approximation $\tilde{X}_{t,x}(t+h)$ has the same mean square order of accuracy as the method based on $\bar{X}_{t,x}(t+h)$, i.e., its order is equal to $p=p_2-1/2$.
\end{tho}

Generally speaking, when implementing implicit numerical methods, the truncated random variables $\Delta \widehat{W}_r(h)$ for the Brownian increments $\Delta W_r(h)=W_r(t+h)-W_r(t),\,r=1,\ldots,D,$ need to be introduced (see \cite{MilsteinMS}). For this end, one can represent $\Delta W_r(h):=\sqrt h\xi_r,\,r=1,\ldots,D,$ where $\xi_r,\,r=1,\ldots,D,$ are independent $N(0,1)$-distributed random variables. Then, one can define $\Delta \widehat{W}_r(h)=\sqrt h\zeta_{rh}$ as follows:
\begin{equation} \label{zetah}
\zeta_{rh}=\left\{\begin{array}{ll}
\xi_r,&\text{if}\quad|\xi_r|\leq A_h, \\
A_h,&\text{if}\quad\xi_r> A_h, \\
-A_h,&\text{if}\quad \xi_r<-A_h,
\end{array}
\right.
\end{equation}
with $A_h:=\sqrt{2k|\ln h|}$, where $k$ is an arbitrary positive integer. The following properties  hold for the truncated Brownian increments.
%\cite[Lemma 1]{MilsteinMS}
\begin{lem} [see~\cite{MilsteinMS}]\label{lem:xi}
 Let $A_h:=\sqrt{2k|\ln h|},\,k\geq1$, and $\zeta_{rh}$ be defined by (\ref{zetah}). Then it holds that
\begin{gather}
   E(\zeta_{rh}-\xi_r)^2\leq h^k,  \\
   0\leq E(\xi_r^2-\zeta_{rh}^2)=1-E\zeta_{rh}^2\leq(1+2\sqrt{2k|\ln h|})h^k.
\end{gather}
\end{lem}
Moreover, it is not difficult to obtain the following properties
\begin{gather}
E\left(|\Delta \widehat{W}_r(h)|^{2p}\right)^{\frac{1}{2p}}\leq E\left(|\Delta{W}_r(h)|^{2p}\right)^{\frac{1}{2p}}\leq c_ph^{1/2}, \quad \forall~ p\in\mathbb{N}^+,\nonumber\\
E(\Delta \widehat{W}_r(h))^{2p-1}=E\left(\Delta {W}_r(h)\right)^{2p-1}=0, \quad \forall~ p\in\mathbb{N}^+,\nonumber\\
E\left(\dw_i\dw_j\dw_k\right)=E\left(\Delta W_i\Delta W_j\Delta W_k\right)=0,\quad \forall~i,j,k\in\{1,2,\ldots,D\},
\label{dw}
\end{gather}
where $c_p$ is a constant independent of $h$.

\section{MAVF methods for stochastic SDEs}\label{S3}
\subsection{MAVF methods for conservative SDEs with single noise}
In this part, we propose MAVF methods preserving  multiple invariants for conservative SDEs with single noise  and prove these methods are of mean square order $1$.

Consider the following autonomous $m$-dimensional SDE with single noise
\begin{equation}\label{SDE2}
 \mathrm{d}Y(t)=f\big(Y(t)\big)\mathrm{d}t+g\big(Y(t)\big)\circ\mathrm{d}W(t),\quad 0\leq t\leq T,\quad Y(0)=Y_0,
\end{equation}
where $f$ and $g$ satisfy the global Lipschitz condition. Let $L(y):\mathbb{R}^m\to \mathbb{R}^\nu$ be the invariant of (\ref{SDE2}), i.e., $ \nabla L(y)f(y)=\nabla L(y)g(y)=\bm{0}$,  for all $y\in \mathbb{R}^m$.

We consider the numerical approximation for \eqref{SDE2} in interval $[0,T]$. Let $0=t_0<t_1<\ldots<t_{N-1}<t_{N}=T$ be a partition of interval $[0,T]$, where $t_n=nh,~n=0,1,\ldots,N$. Let $\{y_n\}_{n=0}^{N}$ be some numerical discretization. We denote $ y_{n+1}=y_{t_n,y_n}(t_{n+1})$, $n=0,1,\ldots,N-1$. For convenience, we write the one-step approximation  as $\bar Y=\bar Y_{t,y}(t+h)$.
Next, we give the MAVF method for (\ref{SDE2}). It is the following one-step approximation $\bar Y$

\begin{equation}\label{method1}
\left\{
\begin{split}
   &\bar Y=y+h\left[\int_0^1f(\sigma(\tau))\,\dtau-\int_0^1\nabla L(\sigma(\tau))^\top\,\dtau\,\alpha_0\right]+\dw\left[\int_0^1g(\sigma(\tau))\,\dtau-\int_0^1\nabla L(\sigma(\tau))^\top\,\dtau\,\alpha_1\right] ,\\
   &\left[\int_0^1\nabla L(\sigma(\tau))\,\dtau\int_0^1\nabla L(\sigma(\tau))^\top\,\dtau\right]\alpha_0=\int_0^1\nabla L(\sigma(\tau))\,\dtau\int_0^1 f(\sigma(\tau))\,\dtau,\\
   &\left[\int_0^1\nabla L(\sigma(\tau))\,\dtau\int_0^1\nabla L(\sigma(\tau))^\top\,\dtau\right]\alpha_1=\int_0^1\nabla L(\sigma(\tau))\,\dtau\int_0^1g(\sigma(\tau))\,\dtau,
\end{split}
\right.
\end{equation}
where $\sigma(\tau)=y+\tau(\bar Y-y)$, $\Delta \widehat W=\sqrt{h}\zeta_h$ and $\zeta_h$ is defined by (\ref{zetah}) with $A_h=\sqrt{4|\ln h|} ~(i.e.,~k=2)$, and $\alpha_0$, $\alpha_1$ are $\mathbb R^\nu$-valued random variables.

As is seen above, the MAVF method in (\ref{method1}) can be regarded as the modification of AVF method in \cite{Ccc}. Here, $\int_0^1\nabla L(\sigma(\tau))^\top\,\dtau\,\alpha_0$ and $\int_0^1\nabla L(\sigma(\tau))^\top\,\dtau\,\alpha_1$ are called modification terms. Let $\bm{\alpha}=(\alpha_0,\alpha_1)$ and we call $\bm{\alpha}$ the modification coefficient of the method (\ref{method1}). The modification coefficient $\bm{\alpha}$ satisfies the second and the third equality in (\ref{method1}) to make the MAVF method conservative.
\subsubsection{The prior estimate of the modification coefficient}
In this section, we give the estimate of high-order moments for the modification coefficient $\bm{\alpha}$. Firstly we obtain the solvability of MAVF method (\ref{method1}) as follows.
\begin{lem}\label{lem:L1}
Suppose that $f(y),g(y) \in \mathbf C^1(\mathbb R^m,\mathbb R^{m})$, $\nabla L(y) \in \mathbf C^1(\mathbb R^m,\mathbb R^{\nu\times m})$, and that $\nabla L(y)\nabla L(y)^\top$ is invertible for arbitrary $y\in\mathbb R^m$. Then for arbitrary given $y \in \mathbb R^m$, the method (\ref{method1}) is uniquely solvable with respect to $\bar Y$ and $\bm{\alpha}$, a.s., for sufficiently small stepsize $h$. Moreover, for every $\epsilon>0$, there exists $h_0>0$ such that for all $h\leq h_0$,
\begin{equation}\label{3.3}
  |\bm{\alpha}|\leq\epsilon,\quad a.s.
\end{equation}
\end{lem}
\begin{proof}
Define
\begin{eqnarray*}
  F(\bar Y,\bm{\alpha},h,\Delta \widehat W)
  =\left(
  \begin{split}
     & F_1(\bar Y,\bm{\alpha},h,\Delta \widehat W) \\
     & F_2(\bar Y,\bm{\alpha},h,\Delta \widehat W)\\
     & F_3(\bar Y,\bm{\alpha},h,\Delta \widehat W)
  \end{split}
  \right),\\
\end{eqnarray*}
where 
\begin{align*}
F_1=&\bar Y-y-h\left[\int_0^1f(\sigma(\tau))\mathrm d\tau-\int_0^1\nabla L(\sigma(\tau))^\top\mathrm d\tau~\alpha_0\right]\\
&-\dw\left[\int_0^1g(\sigma(\tau))\mathrm d\tau-\int_0^1\nabla L(\sigma(\tau))^\top\mathrm d\tau~\alpha_1\right],\\
F_2=&\left[\int_0^1\nabla L(\sigma(\tau))\mathrm d\tau\int_0^1\nabla L(\sigma(\tau))^\top\mathrm d\tau\right]\alpha_0-\int_0^1\nabla L(\sigma(\tau))\mathrm d\tau\int_0^1f(\sigma(\tau))\mathrm d\tau,\\
F_3=&\left[\int_0^1\nabla L(\sigma(\tau))\mathrm d\tau\int_0^1\nabla L(\sigma(\tau))^\top\mathrm d\tau\right]\alpha_1-\int_0^1\nabla L(\sigma(\tau))\mathrm d\tau\int_0^1g(\sigma(\tau))\mathrm d\tau.
\end{align*}
%  &=&\left(
%  \begin{split}
%    &\bar Y-y-h\left[\int_0^1f(\sigma(\tau))\mathrm d\tau-\int_0^1\nabla L(\sigma(\tau))^\top\mathrm d\tau~\alpha_0\right]-\dw\left[\int_0^1g(\sigma(\tau))\mathrm d\tau-\int_0^1\nabla L(\sigma(\tau))^\top\mathrm d\tau~\alpha_1\right]\\
%    & \left[\int_0^1\nabla L(\sigma(\tau))\mathrm d\tau\int_0^1\nabla L(\sigma(\tau))^\top\mathrm d\tau\right]\alpha_0-\int_0^1\nabla L(\sigma(\tau))\mathrm d\tau\int_0^1f(\sigma(\tau))\mathrm d\tau\\
%    & \left[\int_0^1\nabla L(\sigma(\tau))\mathrm d\tau\int_0^1\nabla L(\sigma(\tau))^\top\mathrm d\tau\right]\alpha_1-\int_0^1\nabla L(\sigma(\tau))\mathrm d\tau\int_0^1g(\sigma(\tau))\mathrm d\tau
%  \end{split}
%  \right).
%\end{eqnarray*}
%\end{scriptsize}
Then (\ref{method1}) can be rewritten as $F(\bar Y,\bm{\alpha},h,\Delta \widehat W)=\mathbf 0$. Note that $\sigma(\tau)=y$ provided that $\bar Y=y$. In addition, $\nabla L(y)f(y)=\nabla L(y)g(y)=\bm{0}$ by the definition of invariant $L(y)$. In this way, we obtain
\begin{equation}\label{3.4}
  F(y,\mathbf 0,0,0)=\left(
  \begin{split}
   & \mathbf 0 \\
    -\nabla L(&y)f(y)\\
    -\nabla L(&y)g(y)
  \end{split}
  \right)=\mathbf 0.
\end{equation}
Then, it holds that
\begin{eqnarray*}
%&&\left.\frac{\partial F}{\partial(\bar Y,\alpha)}\right|_{(y,\mathbf 0,0,0)}\\
\left.\frac{\partial F}{\partial(\bar Y,\bm{\alpha)}}\right|_{(y,\mathbf 0,0,0)}&=&\left.\left(
\begin{array}{ccc}
   \frac{\partial F_1}{\partial \bar Y}&\frac{\partial F_1}{\partial \alpha_0}&\frac{\partial F_1}{\partial \alpha_1}\\
   \frac{\partial F_2}{\partial \bar Y}&\frac{\partial F_2}{\partial \alpha_0}&\frac{\partial F_1}{\partial \alpha_1}\\
   \frac{\partial F_3}{\partial \bar Y}&\frac{\partial F_3}{\partial \alpha_0}&\frac{\partial F_3}{\partial \alpha_1}
\end{array}
\right)\right|_{(y,\mathbf 0,0,0)}\\
&=&\left(
\begin{array}{ccc}
  I & \mathbf 0 & \mathbf 0 \\
  \left.\frac{\partial F_2}{\partial \bar Y}\right|_{(y,\mathbf 0,0,0)}& \nabla L(y)\nabla L(y)^\top& \mathbf 0 \\
  \left.\frac{\partial F_3}{\partial \bar Y}\right|_{(y,\mathbf 0,0,0)}& \mathbf 0&\nabla L(y)\nabla L(y)^\top
\end{array}
\right).
\end{eqnarray*}
Using the fact
$$\det\left(\left[
\begin{array}{cc}
A_{m\times m} & \mathbf 0_{m\times n} \\
C_{n\times m} & D_{n\times n}
\end{array}\right]
\right)
=\det(A)\det(D),$$
we have 
\begin{equation}\label{3.5}
\det\left(\left.\frac{\partial F}{\partial(\bar Y,\bm{\alpha)}}\right|_{(y,\mathbf 0,0,0)}\right)=\left|\det(\nabla L(y)\nabla L(y)^\top)\right|^2\neq0.
\end{equation}
In addition, $\frac{\partial F}{\partial(\bar Y,\bm\alpha)}$ is continuous in some neighbourhood of the point $(y,\mathbf 0,0,0)$ by assumptions. We have, by implicit function theorem, that there exists a continuous implicit function $(\bar Y,\bm\alpha)=\left(\bar Y(h,\Delta \widehat W),\bm\alpha(h,\Delta\widehat W)\right)$ defined on some neighbourhood $U(O)$ of $O=(0,0)$. Moreover, $(\bar Y,\bm\alpha)$ satisfies that $\bar Y(0,0)=y$, $\bm\alpha(0,0)=\mathbf 0$ and $F(\bar Y(h,\Delta \widehat W),\bm\alpha(h,\Delta\widehat W),h,\Delta\widehat{W})=\mathbf{0}$, for all $(h,\Delta\widehat W)\in U(O)$.

By the definition of $\Delta \widehat W$, $|\Delta \widehat W|\leq\sqrt h A_h= \sqrt{4h|\ln h|}$, a.s. Noting that $h|\ln h|\to 0,$ as $h\to 0$, we have that for arbitrary  $~\delta >0$, there exists $h_0>0$     such that for all $h\leq h_0$
\begin{equation}\label{3.6}
  |h|\leq\delta,~|\Delta\widehat W|\leq\delta,\quad a.s.
\end{equation}
Thus, for sufficiently small $h$, $(h,\Delta\widehat W)\in U(O),~a.s.$, which indicates that $\bar Y$ and $\bm\alpha $ are uniquely determined by (\ref{method1}). At last, the continuity of $\bm\alpha(h,\Delta \widehat W)$, $\bm\alpha(0,0)=\mathbf0$ and (3.6) yield  \eqref{3.3}.
\end{proof}
\begin{rem}
As is seen in (\ref{3.3}), we actually have that for sufficiently small $h$ independent of $\omega$,  $|\bm\alpha|\leq \epsilon$, a.s., which is essential to give the high-order-moment estimates of the modification coefficient $\bm{\alpha}$. The key to the proof of the boundedness of $\bm{\alpha}$ is the usage of  truncated Brownian increments $|\dw|$. Otherwise, one can only obtain that for every $\epsilon > 0$ and every $\omega\in\Omega$, there exists $h_0(\omega)$ such that for all $h(\omega)\leq h_0(\omega)$, $|\bm\alpha|\leq\epsilon$.
%the $h_0(\omega)$ such that $|\bm\alpha|\leq\epsilon$, for all $h\leq h_0(\omega)$. 
\end{rem}

Next we introduce the Legendre polynomial $\{P_j(t)\}_{j\geq0}$ defined on the interval $[0,1]$. The Legendre polynomial satisfies
\begin{equation*}
  \deg P_i=i,\qquad \int_0^1P_i(t)P_j(t)\,\mathrm d t=\delta_{ij},\qquad \forall~i,j\geq0,
\end{equation*}
where $\delta_{ij}$ is the Kronecker symbol.
Here are some terms of $\{P_j(t)\}_{j\geq0}$:
\begin{equation*}
  P_0(t)\equiv1,\qquad P_1(t)=\sqrt3(2t-1),\qquad P_2(t)=\sqrt5(6t^2-6t+1),~\ldots
\end{equation*}
And it is not hard to obtain the following properties
\begin{equation}\label{3.7}
  \int_0^1P_j(t)\,\mathrm d t=0,~\forall~ j\geq1, \qquad\int_0^1P_j(t)t^k\,\mathrm d t=0,~\forall~ k<j.
\end{equation}
In the following, we will use the properties of Legendre polynomial to derive some important lemmas. The authors in \cite{Brugbook} give Lemma $3.1$ and some facts in Chapter $6$ by means of Legendre polynomials, when dealing with numerical methods for conservative ODEs. Likewise, we  obtain some useful lemmas  in the stochastic cases.
\begin{lem}
\label{lem:L2}
Assume that $f,g$ and $\nabla L$ are continuous, then
\begin{gather}
  \sum_{j\geq0}\int_0^1P_j(\tau)\nabla L(\sigma(\tau))\,\mathrm d\tau\cdot\int_0^1P_j(\tau)f(\sigma(\tau))\,\mathrm d\tau=\bm{0},\nonumber\\
  \sum_{j\geq0}\int_0^1P_j(\tau)\nabla L(\sigma(\tau))\,\mathrm d\tau\cdot\int_0^1P_j(\tau)g(\sigma(\tau))\,\mathrm d\tau=\bm{0}.
  \label{3.8}
\end{gather}
\end{lem}

\begin{proof}
Since the Legendre polynomial $\{P_j(t)\}_{j\geq0}$ forms an orthonormal basis of the Hilbert space $L^2[0,1]$, it follows that $f(\sigma(\tau))=\sum_{j\geq0}P_j(\tau)\left[\int_0^1P_j(\tau)f(\sigma(\tau))\,\mathrm d\tau\right].$ Noting that $\nabla L(y)f(y)=0$, for all $y\in\mathbb R^m$, one has that
\begin{flalign*}
  \bm{0}&= \int_0^1\nabla L(\sigma(\tau))f(\sigma(\tau))\,\mathrm d\tau \\
   &=\int_0^1\nabla L(\sigma(\tau))\cdot\sum_{j\geq0}P_j(\tau)\left[\int_0^1P_j(\tau)f(\sigma(\tau))\,\mathrm d\tau\right]\dtau \\
   &=\sum_{j\geq0}\int_0^1P_j(\tau)\nabla L(\sigma(\tau))\,\mathrm d\tau\cdot\int_0^1P_j(\tau)f(\sigma(\tau))\,\mathrm d\tau.
\end{flalign*}
Likewise, we obtain the second equality of (\ref{3.8}).
\end{proof}
In the sequel, we will use a generic constant $K$, dependent on $y$ but independent of $h$, which may vary from one line to another. 
%And so are the constants in $\mathcal{O}(\cdot)$.
\begin{lem}\label{lem:L3}
Let $G(y)$ be a scalar or vector-valued function defined on $\mathbb{R}^m$ and $(j+1)$ times continuously differentiable with bounded $(j+1)$ derivative. If $f,g$ and $\nabla L$ satisfy the global Lipschitz conditions, then there is a representation
\begin{equation}\label{3.9}
  \int_0^1P_j(t)G(\sigma(\tau))\,\dtau=c_jG^{(j)}(y)(\underbrace{\bar Y-y,\ldots,\bar Y-y}_j)+M_{j,G},\quad \forall~j\geq0,
\end{equation}
where $c_j=\frac{1}{j!}\int_0^1P_j(\tau)\tau^j\,\dtau$, and $[E|M_{j,G}|^{2p}]^{\frac{1}{2p}}=\mathcal{O}(h^{(j+1)/2})$ for all $p=1,2,3,\ldots$
\end{lem}
\begin{proof}
Denote $F:=G\circ\sigma$, then $F$ is $(j+1)$ times continuously differentiable. So Taylor expansion gives
\begin{equation*}
  F(\tau)=\sum_{k=0}^j\frac{F^{(k)}(0)}{k!}\tau^k+
  \int_0^1\frac{(1-\theta)^j}{j!}F^{(j+1)}(\theta\tau)\tau^{j+1}\,\dsit.
\end{equation*}
Noting that $\int_0^1P_j(\tau)\tau^k\,\dtau=0$, for all $k<j$, we obtain
\begin{flalign*}
  &\phantom{111}\int_0^1P_j(t)G(\sigma(\tau))\,\dtau= \int_0^1P_j(t)F(\tau)\,\dtau&\\
   &=\frac{F^{(j)}(0)}{j!}\int_0^1P_j(\tau)\tau^j\,\dtau+
   \int_0^1P_j(\tau)\left[\int_0^1\frac{(1-\theta)^j}{j!}F^{(j+1)}(\theta\tau)\tau^{j+1}\,\dsit\right]\dtau\\ &\triangleq c_jF^{(j)}(0)+M_{j,G},&.
\end{flalign*}
where $c_j=\frac{1}{j!}\int_0^1P_j(\tau)\tau^j\,\dtau$ and $M_{j,G}=\int_0^1P_j(\tau)\left[\int_0^1\frac{(1-\theta)^j}{j!}F^{(j+1)}(\theta\tau)\tau^{j+1}\,\dsit\right]\dtau$.\\
Further, $F^{(k)}(\tau)=G^{(k)}(y+\tau(\bar Y-y))(\underbrace{\bar Y-y,\ldots,\bar Y-y}_{k})$, $k=0,1,\ldots$, leads to
\eqref{lem:L3}.

It remains to estimate the moments of $M_{j,G}$. It follows from the boundedness of $G^{(j+1)}$ that $|M_{j,G}|\leq K_j|\bar Y-y|^{j+1}.$  Recall the first equality of (\ref{method1}).
Since $f,g$ and $\nabla L$ satisfy globally Lipschitz conditions, one is able to prove that
$|\bar Y-y|\leq K(|\dw|+h)$ (This proof is analogous to that of Lemma 2.4 in \cite{MilsteinSmp} ). Using the H\"older inequality and (\ref{dw}), we have
$$[E|M_{j,G}|^{2p}]^{\frac{1}{2p}}=\mathcal{O}(h^{(j+1)/2}),\quad\forall ~p\geq1.$$
This completes the proof.
\end{proof}

Next, we give the prior estimate of the modification coefficient
$\bm{\alpha}$.
\begin{lem}\label{lem:L4}
%Suppose that $f,g$ and $\nabla L$ are Lipschitz continuous.
Suppose that $f,g,~\nabla L\in\mathbf C_b^1(\mathbb R^m)$, $\nabla L\,\nabla L^\top$ is invertible and $[\nabla L\,\nabla L^\top]^{-1}\in \mathbf C_b(\mathbb R^m)$. Then  $\bm\alpha=(\alpha_0,\alpha_1)$ determined by (\ref{method1}) satisfies
\begin{equation}\label{3.10}
  [E|\bm\alpha|^{2p}]^{\frac{1}{2p}}=\mathcal O(h),\qquad |E(\dw\bm\alpha)|=\mathcal O(h^2).
\end{equation}
\end{lem}
\begin{proof}

Firstly, according to the assumptions on $f,g$ and $\nabla L$, we have
\begin{equation}\label{3.11}
\bar Y=y+R_{1,0},\qquad \textrm{with}\qquad [E|R_{1,0}|^{2p}]^{\frac{1}{2p}}=\mathcal O(h^{1/2}).
\end{equation}
Expanding  $f(\sigma(\tau))$
\begin{equation*}
  f(\sigma(\tau))=f(y)+\tau\int_0^1f'(y+\theta\tau(\bar Y-y))(\bar Y-y)\,\dsit.
\end{equation*}
This, associated with (\ref{3.11}) and the assumptions of lemma, yields
\begin{equation}\label{3.12}
  \int_0^1f(\sigma(\tau))\,\dtau=f(y)+R_{1,f},\quad\textrm{with}\quad[E|R_{1,f}|^{2p}]^{\frac{1}{2p}}=\mathcal O(h^{1/2}).
\end{equation}
Similarly, we obtain
\begin{equation}\label{3.13}
  \int_0^1\nabla L(\sigma(\tau))\,\dtau=\nabla L(y)+R_{1,L},\quad\textrm{with}\quad[E|R_{1,L}|^{2p}]^{\frac{1}{2p}}=\mathcal O(h^{1/2}).
\end{equation}
It follows from \eqref{3.13} that the second equation of \eqref{method1} can be written as
\begin{equation}\label{3.14}
  (\nabla L+R_{1,L})(\nabla L^\top+R_{1,L}^\top)\alpha_0=\int_0^1\nabla L(\sigma(\tau))\,\dtau\int_0^1f(\sigma(\tau))\,\dtau.
\end{equation}
Then, by Lemmas \ref{lem:L2} and  \ref{lem:L3}, it holds that
\begin{flalign}\label{3.15}
  &\quad\int_0^1\nabla L(\sigma(\tau))\,\dtau\int_0^1f(\sigma(\tau))\,\dtau &\nonumber\\ &=-\sum_{j\geq1}\int_0^1P_j(\tau)\nabla L(\sigma(\tau))\,\mathrm d\tau\cdot\int_0^1P_j(\tau)f(\sigma(\tau))\,\mathrm d\tau  \nonumber&\\
   &=-\sum_{j\geq 1}[c_j\nabla L^{(j)}(y)(\underbrace{\bar Y-y,\ldots,\bar Y-y}_j)+M_{j,\nabla L}][c_jf^{(j)}(y)(\underbrace{\bar Y-y,\ldots,\bar Y-y}_j)+M_{j,f}] \nonumber&\\
  &=K_1\nabla L'(y)(\bar Y-y)(f'(y)(\bar Y-y))+R_{1,\alpha_0},&
\end{flalign}
where $K_1=-c_1^2$ with $c_1=\int_0^1P_1(\tau)\tau\,\dtau$. And $[E|R_{1,\alpha_0}|^{2p}]^{\frac{1}{2p}}=\mathcal O(h^{1.5}).$\\
Combining (\ref{3.14}) and (\ref{3.15}), we  have
\begin{flalign}\label{3.16}
  \alpha_0=&-[\nabla L\nabla L^\top]^{-1}(\nabla LR_{1,L}^\top+R_{1,L}\nabla L^\top+R_{1,L}R_{1,L}^\top)\alpha_0&\nonumber\\
  &+K_1[\nabla L\nabla L^\top]^{-1}\nabla L'(y)(\bar Y-y)(f'(y)(\bar Y-y))+[\nabla L\nabla L^\top]^{-1}R_{1,\alpha_0}.&
\end{flalign}
By Lemma \ref{lem:L1}, for sufficiently small $h$, $|\alpha_0|\leq1$, a.s. From assumptions on $f$ and $\nabla L$, it follows that
\begin{equation*}
  |\alpha_0|\leq K|R_{1,L}|+K|R_{1,L}|^2+K|\bar Y-y|^2+K|R_{1,\alpha_0}|.
\end{equation*}
This implies that
\begin{equation}\label{3.17}
  [E|\alpha_0|^{2p}]^{\frac{1}{2p}}=\mathcal O(h^{1/2}).
\end{equation}
Then according to (\ref{3.16}) and (\ref{3.17}), we obtain
\begin{flalign*}
  \qquad E|\alpha_0|^2&\leq KE(|R_{1,L}|^2|\alpha_0|^2)+KE|R_{1,L}|^4+KE|\bar Y-y|^4+KE|R_{1,\alpha_0}|^4&\\
  &\leq K[E(|R_{1,L}|^4]^{1/2}[E|\alpha_0|^4]^{1/2}+Kh^2&\\
  &\leq Kh^2.&
\end{flalign*}
That is to say, $[E|\alpha_0|^2]^{1/2}=\mathcal O(h)$. Similarly, we have
\begin{equation}\label{3.18}
  [E|\alpha_0|^{2p}]^{\frac{1}{2p}}=\mathcal O(h).
\end{equation}
Thus (\ref{3.16}) can be rewritten as
\begin{equation*}
  \alpha_0=K_1[\nabla L\nabla L^\top]^{-1}\nabla L'(y)(\bar Y-y)(f'(y)(\bar Y-y))+\widehat R_{1,\alpha_0},
\end{equation*}
with $[E|\widehat R_{1,\alpha_0}|^{2p}]^{\frac{1}{2p}}=\mathcal O(h^{1.5})$.

Because $\bar Y-y=\dw g+\widehat R_{1,0}$ with $[E|\widehat R_{1,0}|^{2p}]^{\frac{1}{2p}}=\mathcal O(h)$, we obtain
\begin{equation*}
  \dw\alpha_0=K_1\dw^3[\nabla L\nabla L^\top]^{-1}\nabla L'g(f'g)+\widetilde R_{1,\alpha_0},\quad\textrm {with}\quad
  [E|\widetilde R_{1,\alpha_0}|^{2p}]^{\frac{1}{2p}}=\mathcal O(h^{2}).
\end{equation*}
Hence, $|E(\dw\alpha_0)|=\mathcal O(h^2)$ due to (\ref{dw}).

As for $\alpha_1$, analogous to the estimate on $\alpha_0$, one can show that
%\begin{small}
\begin{gather*}
  \alpha_1=K_1[\nabla L\nabla L^\top]^{-1}\nabla L'(y)(\bar Y-y)(g'(y)(\bar Y-y))+\widehat R_{1,\alpha_1}, \textrm{with}~ [E|\widehat R_{1,\alpha_1}|^{2p}]^{\frac{1}{2p}}=\mathcal O(h^{1.5}),\\
  \dw\alpha_1=K_1\dw^3[\nabla L\nabla L^\top]^{-1}\nabla L'g(g'g)+\widetilde R_{1,\alpha_1},\quad\textrm {with}\quad
  [E|\widetilde R_{1,\alpha_1}|^{2p}]^{\frac{1}{2p}}=\mathcal O(h^{2}).
\end{gather*}
%\end{small}
Thus, $[E|\alpha_1|^{2p}]^{\frac{1}{2p}}=\mathcal O(h)$ and $|E(\dw\alpha_1)|=\mathcal O(h^2)$.
\end{proof}

\subsubsection{Conservative property and convergence of MAVF methods for SDES with single noise}
\begin{tho}\label{tho:1}
%Assume that $f,g$ and $\nabla L$ are Lipschitz continuous. 
Let $f\in \mathbf C_b^2(\mathbb R^m,\mathbb R^m)$, $g\in \mathbf C_b^3(\mathbb R^m,\mathbb R^m)$ and $\nabla L\in\mathbf C_b^1(\mathbb R^m,\mathbb R^{\nu\times m})$.  Assume that $\nabla L\,\nabla L^\top$ is invertible and $[\nabla L\,\nabla L^\top]^{-1}\in \mathbf C_b^1(\mathbb R^m,\mathbb R^{\nu\times\nu})$. Then the numerical method (\ref{method1}) for SDE (\ref{SDE2}) possesses the following properties:

(\uppercase\expandafter{\romannumeral1}) It preserves multiple invariants $L^i,\,i=1,\ldots,\nu,$ of (\ref{SDE2}), i.e., $L(\bar Y)=L(y)$.

(\uppercase\expandafter{\romannumeral2}) It is of mean square order $1$.
\end{tho}
\begin{proof}
(\uppercase\expandafter{\romannumeral1}) By Taylor expansion and (\ref{method1}), it follows that
\begin{flalign}\label{3.19}
  &L(\bar Y)-L(y)&\nonumber\\
  =&\int_0^1\nabla L(y+\tau(\bar Y-y))(\bar Y -y)\,\dtau &\nonumber\\
  =&h\left[\int_0^1\nabla L(\sigma(\tau))\,\dtau\int_0^1f(\sigma(\tau))\,\dtau - \int_0^1\nabla L(\sigma(\tau))\,\dtau \int_0^1\nabla L(\sigma(\tau))^\top\,\dtau\,\alpha_0 \right]+&\nonumber\\
  \phantom{=}&\dw\left[\int_0^1\nabla L(\sigma(\tau))\,\dtau\int_0^1g(\sigma(\tau))\,\dtau - \int_0^1\nabla L(\sigma(\tau))\,\dtau \int_0^1\nabla L(\sigma(\tau))^\top\,\dtau\,\alpha_1 \right]&\nonumber\\
  =&0.&
\end{flalign}
(\uppercase\expandafter{\romannumeral2}) We divide the proof of this part into four steps:

Step 1: Taylor expansions lead to the results which are shown in formulas \eqref{3.11}, \eqref{3.12} and \eqref{3.13}.
Similarly, we have 
\begin{equation}\label{intg}
  \int_0^1g(\sigma(\tau))\dtau=g(y)+R_{1,g}\quad\textrm{with}\quad[E|R_{1,g}|^{2p}]^{\frac{1}{2p}}=\mathcal O(h^{1/2}),
\end{equation}
In order to evaluate these remainders more precisely, we make a further expansion:
\begin{equation}\label{3.24}
  R_{1,0}=\dw g+\widehat R_{1,0}\quad\textrm{with}\quad[E|\widehat R_{1,0}|^{2p}]^{\frac{1}{2p}}=\mathcal O(h).
\end{equation}
\begin{eqnarray}
   &R_{1,f}&=\int_0^1\tau\int_0^1f'(y+\theta\tau(\bar Y-y))(\bar Y-y)\,\dsit\dtau\nonumber\\
   &\quad&=\frac{1}{2}f'(y)(\bar Y-y)+R^{(1)}_{1,f}\nonumber\\
   &\quad&=\frac{1}{2}\dw f'g+R^{(2)}_{1,f},
   \label{3.25}
\end{eqnarray}
where $[E|R^{(1)}_{1,f}|^{2p}]^{\frac{1}{2p}}=\mathcal O(h)$ and $[E|R^{(2)}_{1,f}|^{2p}]^{\frac{1}{2p}}=\mathcal O(h)$.\\
Similarly, we have
\begin{equation}\label{r1g}
   R_{1,g}=\frac{1}{2}g'(y)(\bar Y-y)+R^{(1)}_{1,g}=\frac{1}{2}\dw g'g+R^{(2)}_{1,g},
\end{equation}
where $[E|R^{(1)}_{1,g}|^{2p}]^{\frac{1}{2p}}=\mathcal O(h)$ and $[E|R^{(2)}_{1,g}|^{2p}]^{\frac{1}{2p}}=\mathcal O(h)$.

Step 2: Formulas \eqref{3.11}-\eqref{3.13}, \eqref{intg}-\eqref{r1g}, Lemma \ref{lem:L4} and H\"older inequality imply that the first equation of \eqref{method1} can be written as
\begin{equation}\label{3.26}
  \begin{split}
    \bar Y &=y+h[f+R_{1,f}-\nabla L(y)^\top\alpha_0-R_{1,L}^\top\alpha_0] \\
      &\phantom{=}+\dw[g+R_{1,g}-\nabla L(y)^\top\alpha_1-R_{1,L}^\top\alpha_1]\\
      &=y+\dw g+hf+R_{2,0},
  \end{split}
\end{equation}
with
\begin{equation}\label{3.27}
  R_{2,0}=\dw R_{1,g}+R^{(1)}_{2,0}=\frac{1}{2}\dw^2g'g+R^{(2)}_{2,0},
\end{equation}
where $[E|R^{(1)}_{2,0}|^{2p}]^{\frac{1}{2p}}=\mathcal O(h^{1.5})$ and $[E|R^{(2)}_{2,0}|^{2p}]^{\frac{1}{2p}}=\mathcal O(h^{1.5})$.

Applying Taylor expansion to $g$ gives:
\begin{equation}\label{3.28}
  g(\sigma(\tau))=g+\tau g'(y)(\bar Y-y)+\tau ^2\int_0^1(1-\theta)g''(y+\theta\tau(\bar Y-y))(\bar Y-y,\bar Y-y)\dsit.
\end{equation}
Thus we have
\begin{eqnarray}\label{3.29}
 &\quad&\int_0^1g(\sigma(\tau))\,\dtau\nonumber\\
 &=&g+\frac{1}{2}g'(y)(\bar Y-y)+\int_0^1\tau ^2\int_0^1(1-\theta)g''(y+\theta\tau(\bar Y-y))(\bar Y-y,\bar Y-y)\dsit\dtau  \nonumber \\
 &=&g+\frac{1}{2}\dw g'g+R_{2,g}.
\end{eqnarray}
Using \eqref{3.26} and \eqref{3.27}, we obtain
\begin{equation}\label{3.30}
  \begin{split}
    R_{2,g} &=\frac{1}{2}hg'f+\frac{1}{2}g'R_{2,0}+\int_0^1\tau ^2\int_0^1(1-\theta)g''(y)(\bar Y-y,\bar Y-y)\dsit\dtau +R^{(1)}_{2,g}\\
      &=\frac{1}{2}hg'f+\frac{1}{4}\dw^2g'(g'g)+\frac{1}{6}\dw^2g''(g,g)+R^{(2)}_{2,g},
  \end{split}
\end{equation}
where $[E|R^{(1)}_{2,g}|^{2p}]^{\frac{1}{2p}}=\mathcal O(h^{1.5})$ and $[E|R^{(2)}_{2,g}|^{2p}]^{\frac{1}{2p}}=\mathcal O(h^{1.5})$.

Step 3: Submitting (\ref{3.12}), (\ref{3.13}) and (\ref{3.29}) into the first equation of (\ref{method1}) leads to
\begin{eqnarray}\label{3.31}
  \bar Y &=&y+h[f+R_{1,f}-\nabla L^\top\alpha_0-R_{1,L}^\top\alpha_0] \nonumber \\
   &\phantom{=}&+\dw[g+\frac{1}{2}\dw g'g+R_{2,g}-\nabla L^\top\alpha_1-R_{1,L}^\top\alpha_1] \nonumber\\
   &=&y+\dw g+hf+\frac{1}{2}\dw^2g'g+R_{3,0},
\end{eqnarray}
where
\begin{eqnarray}
  R_{3,0} &=&\dw R_{2,g}-\dw\nabla L^\top\alpha_1+hR_{1,f}+R^{(1)}_{3,0} \nonumber\\
   &=&\frac{1}{2}h\dw g'f+\frac{1}{4}\dw^3g'(g'g)+\frac{1}{6}\dw^3g''(g,g) \nonumber\\
   &\phantom{=}&-\nabla L^\top\dw\alpha_1+\frac{1}{2}h\dw f'g+R^{(2)}_{3,0},
\end{eqnarray}
with $[E|R^{(1)}_{3,0}|^{2p}]^{\frac{1}{2p}}=\mathcal O(h^{2})$ and $[E|R^{(2)}_{3,0}|^{2p}]^{\frac{1}{2p}}=\mathcal O(h^{2}).$\\

It follows from Lemma \ref{lem:L4} and (\ref{dw}) and  H\"older inequality that
\begin{equation}\label{3.34}
  [E|R_{3,0}|^2]^{1/2}=\mathcal O(h^{1.5})\quad \text{and} \quad |E(R_{3,0})|=\mathcal O(h^2).
\end{equation}

Step $4$: Consider the one-step Milstein approximation $\bar Y^{[M]}$ for (\ref{SDE2}):
\begin{equation}\label{3.35}
  \bar Y^{[M]}=y+\Delta W g+hf+\frac{1}{2}\Delta W^2g'g.
\end{equation}
As is well known, the Milstein method (\ref{3.35}) satisfies the Theorem \ref{tho:MS} with $p_1=1.5,\,p_2=2$.
Comparing our method (\ref{3.31}) with Milstein method, we get
 \begin{equation}\label{3.36}
   \begin{split}
     \bar Y-\bar Y^{[M]}&=(\dw-\Delta W)g+\frac{1}{2}(\dw^2-\Delta W^2)g'g+R_{3,0} \\
       &=\sqrt h(\zeta_h-\xi)g+\frac{1}{2}h(\zeta_h^2-\xi^2)g'g+R_{3,0}.
   \end{split}
 \end{equation}
Owing to Lemma \ref{lem:xi}, $|E(\zeta_h^2-\xi^2)|=\mathcal O(h^{2-\epsilon}),~\forall~ \epsilon\in(0,1)$, and $E(\zeta_h-\xi)^2=\mathcal O(h^2)$. Moreover, one can prove
$E(\zeta_h-\xi)^4=\mathcal O(h^2)$.
H\"older inequality and above evaluations lead to
\begin{equation}\label{3.37}
  E|\bar Y-\bar Y^{[M]}|^2=\mathcal O(h^3),\quad |E(\bar Y-\bar Y^{[M]})|=\mathcal O(h^2).
\end{equation}
Thus the proof of  (\uppercase\expandafter{\romannumeral2}) is completed by the Theorem \ref{tho:MSC}.
\end{proof}
\subsection{ MAVF methods for conservative SDEs with multiple noises}
In this section, we propose the MAVF methods for conservative SDEs with multiple noises and prove that these methods are of mean square oder $1$ if noises are commutative. 

We still suppose that $L(y):\mathbb R^m\to \mathbb R^\nu$ is the invariant of ({\ref{SDE1}}). Based on the ideas of dealing with single noise, we construct the MAVF method for ({\ref{SDE1}}) as follows:
%\begin{small}
\begin{equation}\label{method2}
\left\{
\begin{split}
   &\bar Y=y+h\left[\int_0^1f(\sigma(\tau))\,\dtau-\int_0^1\nabla L(\sigma(\tau))^\top\,\dtau\,\alpha_0\right]
   %&\phantom{\barY=}
   +\sum_{r=1}^D\dw_r\left[\int_0^1g_r(\sigma(\tau))\,\dtau-\int_0^1\nabla L(\sigma(\tau))^\top\,\dtau\,\alpha_r\right], \\
   &\left[\int_0^1\nabla L(\sigma(\tau))\,\dtau\int_0^1\nabla L(\sigma(\tau))^\top\,\dtau\right]\alpha_0=\int_0^1\nabla L(\sigma(\tau))\,\dtau\int_0^1 f(\sigma(\tau))\,\dtau,\\
   &\left[\int_0^1\nabla L(\sigma(\tau))\,\dtau\int_0^1\nabla L(\sigma(\tau))^\top\,\dtau\right]\alpha_r=\int_0^1\nabla L(\sigma(\tau))\,\dtau\int_0^1g_r(\sigma(\tau))\,\dtau,\quad r=1,...,D,
\end{split}
\right.
\end{equation}
%\end{small}
where $\sigma(\tau)=y+\tau(\bar Y-y)$ and $\dw_r=\sqrt h \zeta_{rh}$ is defined (\ref{zetah}) with $k=2$. In addition, $\alpha_r,\,r=0,1,\ldots,D$, are $\mathbb R^{\nu}$-valued random variables. And $\bm\alpha=(\alpha_0,\alpha_1,\ldots,\alpha_D)$ is called the modification coefficient.

\begin{tho}\label{tho:2}
%Assume that $f$, $g_r,\,r=1,\ldots,D$, and $\nabla L$ are Lipschitz continuous,
Let $f\in \mathbf C_b^2(\mathbb R^m,\mathbb R^m)$, $g_r\in \mathbf C_b^3(\mathbb R^m,\mathbb R^m)$, $r=1,\ldots,D$, and $\nabla L\in\mathbf C_b^1(\mathbb R^m,\mathbb R^{\nu\times m})$.  Assume that $\nabla L\nabla L^\top$ is invertible and $[\nabla L\nabla L^\top]^{-1}\in \mathbf C_b^1(\mathbb R^m,\mathbb R^{\nu\times\nu})$. If the noises of SDE \eqref{SDE1} satisfy the commutative conditions, i.e., $g_r^\prime g_i=g_i^\prime g_r,\,i,r=1,\ldots,D$, then the numerical method (\ref{method2}) for SDE (\ref{SDE1}) possesses the following properties:

(\uppercase\expandafter{\romannumeral1}) It preserves multiple invariants $L^i,\,i=1,\ldots,\nu,$ of (\ref{SDE1}), i.e., $L(\bar Y)=L(y)$.

(\uppercase\expandafter{\romannumeral2}) It is of mean square order $1$.
\end{tho}
\begin{proof}
Given that the proof is similar to that of Theorem \ref{tho:1}, we will only give the sketch. The first property (\uppercase\expandafter{\romannumeral1}) easily comes out as  in the proof of Theorem \ref{tho:1}.

%As to the second one, one can establish four parallel lemmas as before under the assumptions. Then compare the method (\ref{method2}) with Milstein method for SDE by Taylor expansion. According to the Theorem \ref{tho:MSC}, one can derive the property (\uppercase\expandafter{\romannumeral2}).
Let us proceed to the proof of (\uppercase\expandafter{\romannumeral2}).
First, as is done in Lemma \ref{lem:L1}, we acquire the solvability of (\ref{method2}) and have that for every $ \epsilon\geq0$, there exists $h_0 ~\textrm{such that}$ for all $h\leq h_0$,
$$|\bm\alpha|\leq\epsilon\quad a.s.$$

Then, analogous to the Lemma \ref{lem:L4}, we have that
\begin{equation*}
  [E|\bm\alpha|^{2p}]^{\frac{1}{2p}}=\mathcal O(h),\qquad |E(\dw_r\bm\alpha)|=\mathcal O(h^2),\,r=1,\ldots,D.
\end{equation*}
Further, using Taylor expansion repeatedly, we acquire
\begin{equation}\label{M1}
  \bar Y=y+\sum_{r=1}^D\dw_rg_r+\frac{1}{2}\sum_{r=1}^D\sum_{i=1}^D\dw_r\dw_ig_r'g_i+hf+R,
\end{equation}
where
\begin{eqnarray*}
  R&=&\frac{1}{2}h\sum_{r=1}^D\dw_rf'g_r+\frac{1}{2}h\sum_{r=1}^D\dw_rg_r'f
  -\sum_{r=1}^D\dw_r\nabla L^\top\alpha_r\nonumber\\
  &\phantom{=}&+\frac{1}{6}\sum_{r,i,j=1}^D\dw_r\dw_i\dw_jg_r''(g_i,g_j)
  +\frac{1}{4}\sum_{r,i,j=1}^D\dw_r\dw_i\dw_jg_r'(g_i'g_j)+\tilde R,
\end{eqnarray*}
with $[E|\tilde R|^{2p}]^{\frac{1}{2p}}=\mathcal O(h^2)$.\\
Thus, it follows that
$$[E|R|^2]^{1/2}=\mathcal O(h^{1.5}),\qquad|ER|=\mathcal O(h^2).$$
%Due to the commutative conditions $g_r'g_i=g_i'g_r,\,r,i=1,\ldots,D$, one can show

Note that, in case of the commutative noises, the Milstein method for \ref{SDE1} becomes
\begin{equation}\label{M3}
  \bar Y^{[M]}=y+\sum_{r=1}^D\Delta W_rg_r+\frac{1}{2}\sum_{r=1}^{D}\sum_{i=1}^D\Delta W_i\Delta W_rg_r'g_i+hf.
\end{equation}
Comparing \eqref{M1} and \eqref{M3}, we have
\begin{equation}\label{M4}
  [E|\bar Y-\bar Y^{[M]}|^2]^{1/2}=\mathcal O(h^{1.5}),\qquad|E(\bar Y-\bar Y^{[M]})|=\mathcal O(h^2).
\end{equation}
which means that the method (\ref{method2}) is of mean square order $1$ by the Theorem \ref{tho:MSC}.
\end{proof}
\begin{rem}
It is noted that, without commutative noises, the mean square order of method (\ref{method2}) is only $\frac{1}{2}$. 
%In addition, the assumptions about the coefficients $f, g_r$ as well as $\nabla L$ can be weaken. Readers, interested in it, can refer to the Theorem $3.4$ in \cite{Cohen}.
\end{rem}

\section{Numerical integration}\label{S4}
When the integrals contained in MAVF methods can not be obtained  directly, we need to use numerical integration to approximate the integrals. In this section, we investigate the effect of numerical integration on MAVF methods, including  mean square convergence order and preservation of invariants.

Here, we recall some concepts of numerical integration. Consider the quadrature formula $(c_i,b_i)_{i=1}^M$ on the interval $[0,1]$:
\begin{equation}\label{4.1}
\int_{0}^{1}f(x)\,\mathrm{d}x\approx\sum_{i=1}^{M}b_if(c_i).
\end{equation}
The quadrature formula \eqref{4.1} is said to have order $q$ if it is exact for polynomials of degree  no larger than $q-1$, i.e.,
$$\int_{0}^{1}x^k\,\mathrm{d}x=\sum_{i=1}^{M}b_ic_i^k,\quad k=0,1,\dots,q-1.$$
Here are some examples of quadrature formulas:
\begin{gather}
\int_{0}^{1}f(x)\,\mathrm{d}x\approx\frac{1}{2}[f(0)+f(1)],\label{4.2}\\
\int_{0}^{1}f(x)\,\mathrm{d}x\approx\frac{1}{4}[3f(\frac{1}{3})+f(1)],\label{4.3}\\
\int_{0}^{1}f(x)\,\mathrm{d}x\approx\frac{1}{2}[f(\frac{3-\sqrt{3}}{6})+f(\frac{3+\sqrt{3}}{6})],\label{4.4}\\
\int_{0}^{1}f(x)\,\mathrm{d}x\approx\frac{1}{18}[5f(\frac{5-\sqrt{15}}{10})+8f(\frac{1}{2})+5f(\frac{5+\sqrt{15}}{10})]\label{Q6},
\end{gather}
and their orders are $2,3,4,6$, respectively.

As is well known, if $f^{(q)}\in\mathbf{C}_b$ (the set of bounded and continuous functions) and with $q$ being the order of the quadrature formula \eqref{4.1}, then it holds that 
\begin{equation}\label{4.5}
\int_{0}^{1}f(x)\,\mathrm{d}x=\sum_{i=1}^{M}b_if(c_i)+\rho_qf^{(q)}(\eta),
\end{equation}
where $\eta\in(0,1)$ and $\rho_q$ is independent of $f$. 
Next, we use the numerical integration to approximate the  integrals in \eqref{method2}. The induced numerical method using the quadrature formula \eqref{4.1} is
%\sum_{i=1}^{M}b_if(\sigma(c_i))
%\sum_{i=1}^{M}b_i\nabla L(\sigma(c_i))\top
\begin{small}
	\begin{equation}\label{method3}
	\left\{
	\begin{split}
	&\widetilde Y=y+h\left[\sum_{i=1}^{M}b_if(\sigma(c_i))-\sum_{i=1}^{M}b_i\nabla L(\sigma(c_i))^\top\alpha_0\right]+\sum_{r=1}^D\dw_r\left[\sum_{i=1}^{M}b_ig_r(\sigma(c_i))-\sum_{i=1}^{M}b_i\nabla L(\sigma(c_i))^\top\alpha_r\right], \\
	&\left[\sum_{i=1}^{M}b_i\nabla L(\sigma(c_i))\right]\left[\sum_{i=1}^{M}b_i\nabla L(\sigma(c_i))^\top\right]\alpha_0=\left[\sum_{i=1}^{M}b_i\nabla L(\sigma(c_i))\right]\left[\sum_{i=1}^{M}b_if(\sigma(c_i))\right], \\
	&\left[\sum_{i=1}^{M}b_i\nabla L(\sigma(c_i))\right]\left[\sum_{i=1}^{M}b_i\nabla L(\sigma(c_i))^\top\right]\alpha_r=\left[\sum_{i=1}^{M}b_i\nabla L(\sigma(c_i))\right]\left[\sum_{i=1}^{M}b_ig_r(\sigma(c_i))\right], \quad r=1,...,D,
	\end{split}
	\right.
	\end{equation}
\end{small}
where $\sigma(\tau)=y+\tau(\widetilde Y-y)$.
% and the order $q$ of quadrature formula $(c_i,b_i)_{i=1}^{M}$ satisfies $q\geq2$. 
\subsection{Mean square convergence order}
In this part, we study the mean square convergence order of \eqref{method3}. Following the procedure in Section \ref{S3}, we firstly present the boundedness of $\bm{\alpha}$, and the expansion formula of $\bar{Y}$.
\begin{lem}\label{L4.1}
Let $q\geq 1$ be the order of quadrature formula $(c_i,b_i)_{i=1}^{M}$. Suppose that $f,g_r\in \mathbf C_b^1(\mathbb R^m,\mathbb R^{m})$, $r=1,...,D$, $\nabla L \in \mathbf C_b^1(\mathbb R^m,\mathbb R^{\nu\times m})$, and that $\nabla L\nabla L^\top$ is invertible.  Then for arbitrary given $y \in \mathbb R^m$, the method (\ref{method3}) is uniquely solvable with respect to $\widetilde Y$ and $\bm{\alpha}=(\alpha_0,\ldots,\alpha_D)$, a.s., for sufficiently small stepsize $h$. Moreover, for every $\epsilon>0$, there exists $h_0>0$ such that for all $h\leq h_0$,
\begin{equation}\label{4.7}
|\bm{\alpha}|\leq\epsilon,\quad a.s.
\end{equation}
In addition, there is a representation
\begin{equation}\label{4.8}
\widetilde{Y}=y+\sum_{r=1}^{D}\dw_rg_r+R_{\widetilde{Y}},
\end{equation}
with $\left[E|R_{\widetilde{Y}}|^{2p}\right]^{\frac{1}{2p}}=\mathcal{O}(h)$, $p=1,2,\dots$.
\end{lem}
The proof of Lemma \ref{L4.1} follows from the fact that $\sum_{i=1}^{M}b_i=1$ for $q\geq 1$, and the implicit function theorem, as in the proof of Lemma \ref{lem:L1}. Thus we omit the proof.

The following lemma is used to estimate the accuracy of numerical integration in \eqref{method3}.
\begin{lem}\label{L4.2}
Let  $f,\,g_r \in \mathbf{C}_b^1,~ r=1,\ldots,D$. Let $q$ be the order of quadrature formula $(c_i,b_i)_{i=1}^{M}$ and $G$ be an arbitrary scalar or vector-valued function.  Assume that 
$\nabla L \in \mathbf{C}_b^1$ and $\nabla L\nabla L^\top$ is invertible.  If $q\geq 2$ and $G^{(q)}\in\mathbf C_b$, then we have
\begin{equation}\label{4.9}
\int_{0}^{1}G(\sigma(\tau))\,\dtau=\sum_{i=1}^{M}b_iG(\sigma(c_i))+\Psi_{G,q},
\end{equation} 
where $\Psi_{G,q}$ satisfies that
\begin{equation}\label{4.10}
\left[E|\Psi_{G,q}|^{2p}\right]^{\frac{1}{2p}}=\mathcal{O}(h^{\frac{q}{2}}),\quad \forall~p=1,2,\dots
\end{equation}
In addition, it holds that

(1) If $q$ is odd, then 
\begin{equation}\label{4.11}
\left|E\left(\dw_r\Psi_{G,q}\right)\right|=\mathcal{O}(h^{\frac{q+1}{2}}),\quad r=1,\ldots,D.
\end{equation}

(2) If $q$ is even and $G^{(q+1)}\in\mathbf C_b$, then
\begin{equation}\label{4.12}
\left|E\left(\dw_r\Psi_{G,q}\right)\right|=\mathcal{O}(h^{\frac{q+2}{2}}),\quad r=1,\ldots,D.
\end{equation}
\end{lem}
\begin{proof}
We only prove the case that $G$ is a scalar function, the case of  vector function is analogous. According to \eqref{4.5}, we have
$$\int_{0}^{1}G(\sigma(\tau))\,\dtau=\sum_{i=1}^{M}b_iG(\sigma(c_i))+\Psi_{G,q},$$
where, $\left.\Psi_{G,q}=\rho_q\frac{\mathrm d^q}{\dtau^q}G(\sigma(\tau))\right|_{\tau=\eta}$ with $\eta\in(0,1)$.\\
Besides, it follows that 
\begin{equation}\label{4.13}
\left.\frac{\mathrm d^q}{\dtau^q}G(\sigma(\tau))\right|_{\tau=\eta}=G^{(q)}(\sigma(\eta))(\underbrace{\widetilde Y-y,\ldots,\widetilde Y-y}_q).
\end{equation}
Due to \eqref{4.8} and boundedness of $G^{(q)}$, we have
$$[E|\Psi_{G,q}|^{2p}]^{\frac{1}{2p}}=\mathcal{O}(h^{\frac{q}{2}}).$$
(1) If $q$ is odd,  H\"older inequality yields
$$|E\dw_r\Psi_{G,q}|=\mathcal{O}(h^{\frac{q+1}{2}}),\quad r=1,\ldots,D.$$
(2) In case that $q$ is even, it holds that  
\begin{equation*}
E\left(\dw_{i_1}\dw_{i_2}\cdots\dw_{i_{q+1}}\right)=0, \quad \forall ~i_1,i_2,\ldots,i_{q+1}\in\{1,2,\ldots,D\}.
\end{equation*}
Since $G^{(q+1)}\in\mathbf{C}_b$, we have $G^{(q)}(\sigma(\eta))=G^{(q)}(y)+R_G$, with $[E|R_G|^{2p}]^{\frac{1}{2p}}=\mathcal{O}(h^{1/2})$.
Thus, using H\"older inequality and \eqref{4.8} gives
$$|E\dw_r\Psi_{G,q}|=\mathcal{O}(h^{\frac{q+2}{2}}),\quad r=1,\ldots,D.$$
\end{proof}
Next, we give the prior estimate of modification $\bm{\alpha}$ in \eqref{method3}.
\begin{lem} \label{L4.3}
Let $f,\,g_r,r=1,\ldots,D,\,\nabla L \in \mathbf{C}_b^{(q+1)}$. Let $q$ be the order of quadrature formula $(c_i,b_i)_{i=1}^{M}$ in \eqref{method3}.  Assume that $\nabla L\nabla L^\top$ is invertible, and $\left[\nabla L\nabla L^\top\right]^{-1}\in\mathbf{C}_b^1$. If
$q\geq2$,
then we have
\begin{equation}\label{4.14}
\left[E|\bm\alpha|^{2p}\right]^{\frac{1}{2p}}=\mathcal O(h),~p=1,2,\ldots\qquad\text{and}\qquad |E(\dw_r\bm\alpha)|=\mathcal O(h^2),\,r=1,\ldots,D.
\end{equation}
%if one of the following conditions holds
%
%$\mathrm{(a)}$ $q=2$ and $f,g_r,\nabla L\in\mathbf{C}_b^3$;
%
%$\mathrm{(b)}$  $q\geq 3$ and $f^{(q)},g_r^{(q)},\nabla L^{(q)}\in\mathbf{C}_b$.
\end{lem}

\begin{proof}
By \eqref{4.9} in Lemma \ref{L4.2}, the second equation in \eqref{method3} can be rewritten as
%\Psi_{G,q}   [E|\Psi_{G,q}|^{2p}]^{\frac{1}{2p}}=\mathcal{O}(h^{\frac{q}{2}}).
%|E\dw_r\Psi_{G,q}|=\mathcal{O}(h^{\frac{q+1}{2}}),\quad r=1,\ldots,q.
%\int_0^1\nabla L(\sigma(\tau))^\top\,\dtau\
\begin{eqnarray}
&\quad&\left[\int_0^1\nabla L(\sigma(\tau))\,\dtau- \Psi_{\nabla L,q}\right]\left[\int_0^1\nabla L(\sigma(\tau))^\top\,\dtau- \Psi_{\nabla L,q}^\top\right]\alpha_0 \nonumber\\
&=&\left[\int_0^1\nabla L(\sigma(\tau))\,\dtau- \Psi_{\nabla L,q}\right]\left[\int_0^1f(\sigma(\tau))\,\dtau- \Psi_{f,q}\right]. \label{4.15}
\end{eqnarray}
By arranging the above formula, we have
\begin{eqnarray}
\left[\int_0^1\nabla L(\sigma(\tau))\,\dtau \int_0^1\nabla L(\sigma(\tau))^\top\,\dtau\right]\alpha_0= \int_0^1\nabla L(\sigma(\tau))\,\dtau\int_0^1f(\sigma(\tau))\,\dtau+T_{\alpha_0}, \label{4.16}
\end{eqnarray}
where
\begin{eqnarray}
T_{\alpha_0}&=&-\int_0^1\nabla L(\sigma(\tau))\,\dtau\Psi_{f,q}-\Psi_{\nabla L,q}\int_0^1f(\sigma(\tau))\,\dtau+\Psi_{\nabla L,q}\Psi_{f,q}\nonumber\\
&&+\left[\Psi_{\nabla L,q}\int_0^1\nabla L(\sigma(\tau))^\top\,\dtau+\int_0^1\nabla L(\sigma(\tau))\,\dtau\Psi_{\nabla L,q}^\top-\Psi_{\nabla L,q}\Psi_{\nabla L,q}^\top\right]\alpha_0.\nonumber
\end{eqnarray}
Using \eqref{3.12}, \eqref{3.13} and \eqref{3.15}, we write \eqref{4.16} as
\begin{equation*}
(\nabla L+R_{1,L})(\nabla L^\top+R_{1,L}^\top)\alpha_0=K_1\nabla L'(y)(\widetilde Y-y)(f'(y)(\widetilde Y-y))+R_{1,\alpha_0}+T_{\alpha_0},
\end{equation*}
where $[E|R_{1,\alpha_0}|^{2p}]^{\frac{1}{2p}}=\mathcal{O}(h^{1.5})$.\\
Further, since $[\nabla L\nabla L^\top]^{-1}\in\mathbf{C}_b^1$, we have
\begin{flalign}\label{4.17}
\alpha_0=&-[\nabla L\nabla L^\top]^{-1}(\nabla LR_{1,L}^\top+R_{1,L}\nabla L^\top+R_{1,L}R_{1,L}^\top)\alpha_0+[\nabla L\nabla L^\top]^{-1}T_{\alpha_0}&\nonumber\\
&+K_1[\nabla L\nabla L^\top]^{-1}\nabla L'(y)(\widetilde Y-y)(f'(y)(\widetilde Y-y))+[\nabla L\nabla L^\top]^{-1}R_{1,\alpha_0}.&
\end{flalign}
Next we prove the conclusion under the two cases  $q=2$ and $q\geq3$ respectively.\\
(1) Since $q=2$ is even and  $f,g_r,\nabla L\in\mathbf{C}_b^3$, it follows from Lemma \ref{L4.2} that\\
$$\left[E|\Psi_{f,q}|^{2p}\right]^{\frac{1}{2p}}=\mathcal{O}(h),\quad \left[E|\Psi_{\nabla L,q}|^{2p}\right]^{\frac{1}{2p}}=\mathcal{O}(h),$$
and
$$\left|E\left(\dw_r\Psi_{f,q}\right)\right|=\mathcal{O}(h^2),\quad \left|E\left(\dw_r\Psi_{\nabla L,q}\right)\right|=\mathcal{O}(h^2).$$
Hence, $[E|T_{\alpha_0}|^{2p}]^{\frac{1}{2p}}=\mathcal{O}(h)$, and $|E\dw_r T_{\alpha_0}|=\mathcal{O}(h^2)$.
Analogous to the proof of Lemma \ref{lem:L4}, we have
\begin{equation*}
[E|\bm\alpha|^{2p}]^{\frac{1}{2p}}=\mathcal O(h),\qquad |E(\dw_r\bm\alpha)|=\mathcal O(h^2),\quad r=1,\ldots,D.
\end{equation*}
(2) According to Lemma \ref{L4.2}, we obtain
$$[E|\Psi_{f,q}|^{2p}]^{\frac{1}{2p}}=\mathcal{O}(h^{\frac{q}{2}}),\quad [E|\Psi_{\nabla L,q}|^{2p}]^{\frac{1}{2p}}=\mathcal{O}(h^{\frac{q}{2}}).$$
Notice that $\frac{q}{2}\geq1.5$ provided that $q \geq3$.
Then H\"older inequality yields
$$[E|T_{\alpha_0}|^{2p}]^{\frac{1}{2p}}=\mathcal{O}(h^{1.5}), \quad\text{and} \quad|E\dw_r T_{\alpha_0}|=\mathcal{O}(h^2).$$
Combing above formula and \eqref{4.17} produces \eqref{4.14}.
\end{proof}

%\begin{rem}
%In fact, we can give a general condition which covers the two cases (a) and (b) of Lemma \ref{L4.3} to acquire same conclusion.
%We  only need replace (a) and (b) by the condition\\
%$$q\geq2, ~\text{and}~ f,g_r,\nabla L\in\mathbf{C}_b^{(q+1)}.$$ 
%In this sequel, we will use above condition for the convenience of narration.
%\end{rem}
Next we give the result of convergence of method \eqref{method3}.
\begin{tho} \label{tho3}
Let $\nabla L\nabla L^\top$ be invertible and $\left[\nabla L\nabla L^\top\right]^{-1}\in\mathbf{C}_b^1$. Let $q$ be the order of quadrature formula $(c_i,b_i)_{i=1}^{M}$ in \eqref{method3}. Assume that $q\geq2, ~\text{and}~ f,g_r,~r=1,\ldots,D,\,\nabla L\in\mathbf{C}_b^{(q+1)}$. If the noises satisfy the commutative conditions, i.e., $g_r^\prime g_i=g_i^\prime g_r,~i,r=1,\ldots,D$, then the method \eqref{method3} is of mean square order $1$.
\end{tho}
\begin{proof}
This proof is analogous to that of Theorem \ref{tho:2}, we only give the sketch.
It follows from Lemma \ref{L4.2} that
%\Psi_{G,q}   [E|\Psi_{f,q}|^{2p}]^{\frac{1}{2p}}=\mathcal{O}(h^{\frac{q}{2}}).
%|E\dw_r\Psi_{G,q}|=\mathcal{O}(h^{\frac{q+1}{2}}),\quad r=1,\ldots,q.
%\int_0^1\nabla L(\sigma(\tau))^\top\,\dtau
%\int_0^1f(\sigma(\tau))\,\dtau  %\int_0^1g_r(\sigma(\tau))\,\dtau
\begin{flalign}
\widetilde{Y}=&y+h\left[\int_0^1f(\sigma(\tau))\,\dtau-\Psi_{f,q}-\int_0^1\nabla L(\sigma(\tau))^\top\,\dtau\alpha_0+\Psi_{\nabla L,q}^\top\alpha_0 \right]&\nonumber\\
&+\sum_{r=1}^{D}\dw_r\left[\int_0^1g_r(\sigma(\tau))\,\dtau-\Psi_{g_r,q}-\int_0^1\nabla L(\sigma(\tau))^\top\,\dtau\alpha_r+\Psi_{\nabla L,q}^\top\alpha_r \right]&\nonumber\\
=&y+h\left[\int_0^1f(\sigma(\tau))\,\dtau-\int_0^1\nabla L(\sigma(\tau))^\top\,\dtau\alpha_0 \right]+\sum_{r=1}^{D}\dw_r\left[\int_0^1g_r(\sigma(\tau))\,\dtau-\int_0^1\nabla L(\sigma(\tau))^\top\,\dtau\alpha_r\right]+R^{[0]},&
\label{4.18}
\end{flalign}
where
$$R^{[0]}=-h\Psi_{f,q}+h\Psi_{\nabla L,q}^\top\alpha_0+ \sum_{r=1}^{D}\dw_r\left[-\Psi_{g_r,q}+\Psi_{\nabla L,q}^\top\alpha_r\right].$$
According to Lemma \ref{L4.2} and Lemma \ref{L4.3}, we have
$$[E|R^{[0]}|^{2p}]^{\frac{1}{2p}}=\mathcal{O}(h^{1.5}),\quad|ER^{[0]}|=\mathcal{O}(h^2).$$ 
Using Taylor expansion, analogous to the proof of Theorem \ref{tho:2}, we obtain
\begin{equation}\label{4.19}
\widetilde Y=y+\sum_{r=1}^D\dw_rg_r+\frac{1}{2}\sum_{r=1}^D\dw_r^2g_r'g_r
+\sum_{r=1}^{D-1}\sum_{i=r+1}^D\dw_i\dw_rg_r'g_i+hf+T^{[0]}.
\end{equation}
where
\begin{eqnarray}
	T^{[0]}&=&\frac{1}{2}h\sum_{r=1}^D\dw_rf'g_r+\frac{1}{2}h\sum_{r=1}^D\dw_rg_r'f
	-\sum_{r=1}^D\dw_r\nabla L^\top\alpha_r+\frac{1}{6}\sum_{r,i,j=1}^D\dw_r\dw_i\dw_jg_r''(g_i,g_j)\nonumber\\
	&\phantom{=}&
	+\frac{1}{4}\sum_{r,i,j=1}^D\dw_r\dw_i\dw_jg_r'(g_i'g_j)+R^{[0]}+\widehat{R}^{[0]},
\end{eqnarray}
with $E|\widehat{R}^{[0]}|^2=\mathcal{O}(h^2)$.
Thus, it holds that 
$$[E|T^{[0]}|^2]^{1/2}=\mathcal{O}(h^{1.5}),\quad|ET^{[0]}|=\mathcal{O}(h^2).$$
Comparing \eqref{4.19} with Milstein method for SDE \eqref{SDE1}, we obtain that the method \eqref{method3} is of mean square order $1$.
\end{proof}
%Theorem \ref{tho3} shows that the usage of numerical integration will not lower the mean square order of MAVF methods, which makes MAVF methods more practicable. On the other hand, the method \eqref{method3} does not preserve the invariant of origin system due to the usage of quadrature formula. Thus, it is necessary to study the preservation of invariant of \eqref{method3}. At fist, we give the definition of mean square order of invariant conservation.
\subsection{Mean square order of invariant conservation}
It is worth noting that the method \eqref{method3} does not preserve exactly the invariant of original system generally, due to the usage of quadrature formula, which makes it necessary to study the preservation of invariant of \eqref{method3}. In the following, we give the definition of mean square order of invariant conservation.
\begin{Def}
A numerical discretization $\{y_n\}_{n=0}^{N}$ is said to have mean square order $p$ of invariant conservation, if the invariant $ L(y)$ of SDE \eqref{SDE1} satisfies
\begin{equation}\label{4.21}
[E|L(y_N)-L(y_0)|^2]^{1/2}=\mathcal{O}(h^p).
\end{equation}
\end{Def}
%Let $0=t_0<t_1<\ldots<t_{N-1}<t_{N}=T$ be a partition of interval $[0,T]$, where $t_n=nh,~n=0,1,\ldots,N$.
Let $\{y_n\}_{n=0}^{N}$ be the numerical discretization corresponding to the one-step approximation \eqref{method3} with $y_0=Y_0$, and denote $\widetilde Y_{t_n,y_n}(t_{n+1})=y_{n+1}$, $n=0,1,\ldots,N-1$. The numerical method generated from the one-step approximation \eqref{method3} reads
%\begin{small}
	\begin{equation}\label{4.22}
	\left\{
	\begin{split}
	&\widetilde Y_{t_n,y_n}(t_{n+1})=y_n+h\left[\sum_{i=1}^{M}b_if(\sigma_n(c_i))-\sum_{i=1}^{M}b_i\nabla L(\sigma_n(c_i))^\top\alpha_0\right]\\
	&\phantom{\tilde Y=}+\sum_{r=1}^D\dw_{r,n}\left[\sum_{i=1}^{M}b_ig_r(\sigma_n(c_i))-\sum_{i=1}^{M}b_i\nabla L(\sigma_n(c_i))^\top\alpha_r\right], \\
	&\left[\sum_{i=1}^{M}b_i\nabla L(\sigma_n(c_i))\right]\left[\sum_{i=1}^{M}b_i\nabla L(\sigma_n(c_i))^\top\right]\alpha_0=\left[\sum_{i=1}^{M}b_i\nabla L(\sigma_n(c_i))\right]\left[\sum_{i=1}^{M}b_if(\sigma_n(c_i))\right], \\
	&\left[\sum_{i=1}^{M}b_i\nabla L(\sigma_n(c_i))\right]\left[\sum_{i=1}^{M}b_i\nabla L(\sigma_n(c_i))^\top\right]\alpha_r=\left[\sum_{i=1}^{M}b_i\nabla L(\sigma_n(c_i))\right]\left[\sum_{i=1}^{M}b_ig_r(\sigma_n(c_i))\right],\quad r=1,...,D,
	\end{split}
	\right.
	\end{equation}
%\end{small}
where $\sigma_n(\tau)=y_n+\tau(\widetilde Y_{t_n,y_n}(t_{n+1})-y_n))$, and  $\dw_{r,n}=\dw_r(t_{n+1})-\dw_r(t_n)$, $r=1,\dots,D$, $n=0,1,\ldots,N-1$, are mutually independent truncated Brownian increments.\\
The following lemma gives the one-step error estimate of invariant conservation of method \eqref{4.22}.
\begin{lem} \label{L4.4}
Let $\nabla L\nabla L^\top$ be invertible and $\left[\nabla L\nabla L^\top\right]^{-1}\in\mathbf{C}_b^1$. Let $q$ be the order of quadrature formula $(c_i,b_i)_{i=1}^{M}$ in \eqref{method3}. Assume  that $q\geq2, ~\text{and}~ f,g_r,~r=1,\ldots,D,\,\nabla L\in\mathbf{C}_b^{(q+1)}$. Then it holds that	
%\widetilde Y_{t_n,y_n}(t_{n+1})
\begin{equation}\label{4.23}
\left[E| L(\widetilde Y_{t_n,y_n}(t_{n+1}))-L(y_n)|^2\right]^{1/2}=\mathcal{O}(h^{\frac{q+1}{2}}),\quad n=0,1,\ldots,N-1,
\end{equation}
and
\begin{equation}\label{4.24}
	\left|E\left[\left.L(\widetilde{Y}_{t_n,y_n}(t_{n+1}))-L(y_n)\right|\mathcal{F}_{t_n}\right]\right|=\left\{\begin{split}
&\mathcal{O}(h^{\frac{q+2}{2}}),\quad \text{if $q$ is even},\\
&\mathcal{O}(h^{\frac{q+1}{2}}),\quad \text{if $q$ is odd}.
\end{split}\right.
\end{equation}
\end{lem}
\begin{proof}
Let $y_n$ denote the random variable and $y$ denote the deterministic variable in this proof.
It follows from Taylor expansion that
\begin{equation}\label{4.25}
L(\widetilde Y_{t_n,y_n}(t_{n+1}))-L(y_n)=\int_{0}^{1}\nabla L(\sigma_n(\tau))\,\dtau(\widetilde Y_{t_n,y_n}(t_{n+1})-y_n).
\end{equation}
By \eqref{4.9}, we have
\begin{equation}\label{4.26}
\int_{0}^{1}\nabla L(\sigma_n(\tau))\,\dtau=\sum_{i=1}^{M}b_i\nabla L(\sigma_n(c_i))+\Psi_{\nabla L,q},
\end{equation}
with $\left[E|\Psi_{\nabla L,q}|^{2p}\right]^{\frac{1}{2p}}=\mathcal{O}(h^{\frac{q}{2}})$.\\
Submitting \eqref{4.26} and the first equation of \eqref{4.22} into \eqref{4.25} gives
\begin{flalign}
&L(\widetilde Y_{t_n,y_n}(t_{n+1}))-L(y_n)&\nonumber \\
=&h\left[\sum_{i=1}^{M}b_i\nabla L(\sigma_n(c_i))+\Psi_{\nabla L,q}\right]\left[\sum_{i=1}^{M}b_if(\sigma_n(c_i))-\sum_{i=1}^{M}b_i\nabla L(\sigma_n(c_i))^\top\alpha_0\right]&\nonumber\\
&+\sum_{r=1}^{D}\dw_{r,n}\left[\sum_{i=1}^{M}b_i\nabla L(\sigma_n(c_i))+\Psi_{\nabla L,q}\right]\left[\sum_{i=1}^{M}b_ig_r(\sigma_n(c_i))-\sum_{i=1}^{M}b_i\nabla L(\sigma_n(c_i))^\top\alpha_r\right].&\nonumber
\end{flalign}
Utilizing the second and third lines of \eqref{4.22}, we obtain
\begin{flalign}
&L(\widetilde Y_{t_n,y_n}(t_{n+1}))-L(y_n)&\nonumber \\
=&h\Psi_{\nabla L,q}\left[\sum_{i=1}^{M}b_if(\sigma_n(c_i))-\sum_{i=1}^{M}b_i\nabla L(\sigma_n(c_i))^\top\alpha_0\right]&\nonumber\\
&+\sum_{r=1}^{D}\dw_{r,n}\Psi_{\nabla L,q}\left[\sum_{i=1}^{M}b_ig_r(\sigma_n(c_i))-\sum_{i=1}^{M}b_i\nabla L(\sigma_n(c_i))^\top\alpha_r\right].&\label{4.27}
\end{flalign}
In order to acquire \eqref{4.23}, it suffices to estimate the lowest-order term
$$\sum_{r=1}^{D}\dw_{r,n}\Psi_{\nabla L,q}\left[\sum_{i=1}^{M}b_ig_r(\sigma_n(c_i))\right].$$
According to assumptions on $f,~g_r,~\nabla L$ and H\"older inequality, we have 
\begin{eqnarray}
E| L(\widetilde Y_{t_n,y_n}(t_{n+1}))-L(y_n)|^2&\leq&K\sum_{r=1}^{D}E\left[|\dw_{r,n}|^2|\Psi_{\nabla L,q}|^2\right]+Kh^{q+2}\nonumber\\
&\leq&Kh\left[E|\Psi_{\nabla L,q}|^4\right]^{1/2}+Kh^{q+2}\nonumber\\
&\leq&Kh^{q+1}.\label{4.28}
\end{eqnarray} 
Thus this proves \eqref{4.23}.\\
If $y_n$ is replaced by the deterministic variable $y$, we are able to use Lemma \ref{L4.2} and acquire that
\begin{equation}\label{4.29}
|E\dw_r\Psi_{\nabla L,q}|=\left\{
\begin{split}
&\mathcal{O}(h^{\frac{q+2}{2}}),\quad \text{if $q$ is even},\\
&\mathcal{O}(h^{\frac{q+1}{2}}),\quad \text{if $q$ is odd}.
\end{split}\right.
\end{equation}
Thus, we have
\begin{equation*}
\left|E\left[L(\widetilde{Y}_{t_n,y}(t_{n+1}))-L(y)\right]\right.=\left\{
\begin{split}
&\mathcal{O}(h^{\frac{q+2}{2}}),\quad \text{if $q$ is even},\\
&\mathcal{O}(h^{\frac{q+1}{2}}),\quad \text{if $q$ is odd}.
\end{split}\right.
\end{equation*}
Notice that $y_n$ is $\mathcal{F}_{t_n}$-measurable and that $\widetilde Y_{t_n,y}(t_{n+1})-L(y)$  is $\mathcal{F}_{t_n}$-independent.
According to the property of conditional expectation (see \cite[Chapter $1$]{MilsteinBook}), we have
\begin{eqnarray}
		E\left[\left.L(\widetilde{Y}_{t_n,y_n}(t_{n+1}))-L(y_n)\right|\mathcal{F}_{t_n}\right]
	=	\left(E\left[\left.L(\widetilde{Y}_{t_n,y}(t_{n+1}))-L(y)\right]\right)\right|_{y=y_n}. \label{CE}
\end{eqnarray} 
In this way, we obtain \eqref{4.24}.
\end{proof}

We now give the result about mean square order of invariant conversation for \eqref{4.22}.
\begin{tho}\label{tho4}
Let $\nabla L\nabla L^\top$ be invertible and $\left[\nabla L\nabla L^\top\right]^{-1}\in\mathbf{C}_b^1$. Let $q$ be the order of quadrature formula $(c_i,b_i)_{i=1}^{M}$ in \eqref{method3}. Assume  that $q\geq2, ~\text{and}~ f,g_r,~r=1,\ldots,D,\,\nabla L\in\mathbf{C}_b^{(q+1)}$. Then it holds that	
\begin{equation}\label{4.30}
\left[E| L(y_N)-L(y_0)|^2\right]^{1/2}=\left\{
\begin{split}
&\mathcal{O}(h^{\frac{q}{2}}),\quad\quad \text{if $q$ is even},\\
&\mathcal{O}(h^{\frac{q-1}{2}}),\quad \text{if $q$ is odd}.
\end{split}\right.
\end{equation}	
\end{tho}
\begin{proof}
Denote $e_{n}=E|L(y_n)-L(y_0)|^2$, $n=0,1,\dots,N$,
%Since $L(y)$ is the invariant of original equation \eqref{SDE1}, it holds that $L(Y(t_n))=L(y_0)$, $n=0,1,\dots,N$.
and we have
\begin{eqnarray}
e_{n+1}&=&E|L(y_{n+1})-L(y_0)|^2\nonumber\\
&=&E|L(y_{n+1})-L(y_n)+L(y_n)-L(y_0))|^2\nonumber\\
&=&E|L(y_{n+1})-L(y_n)|^2+E|L(y_{n})-L(y_0)|^2\nonumber\\
&&+2E\left[\left(L(y_n)-L(y_0)\right)^\top\left(L(y_{n+1})-L(y_n)\right)\right]\nonumber\\
&=&e_n+E|L(y_{n+1})-L(y_n)|^2+2S,\quad n=0,1,\ldots,N-1,\label{4.31}
\end{eqnarray}
where $S=E\left[\left(L(y_n)-L(y_0)\right)^\top\left(L(y_{n+1})-L(y_n)\right)\right]$.
Since $y_n$ and $y_0$ are $\mathcal{F}_{t_n}$-measurable, it follows that 
\begin{eqnarray}
S&=&E\left(E\left.\left[(L(y_n)-L(y_0))^\top(L(y_{n+1})-L(y_n))\right|\mathcal{F}_{t_n}\right]\right)\nonumber\\
&=&E\left[(L(y_n)-L(y_0))^\top E\left(\left.L(y_{n+1})-L(y_n)\right|\mathcal{F}_{t_n}\right)\right]\nonumber\\
&\leq&\left[E|L(y_n)-L(y_0)|^2\right]^{1/2}\left\{E\left|E\left[\left.L(y_{n+1})-L(y_n)\right|\mathcal{F}_{t_n}\right]\right|^2\right\}^{1/2}\nonumber\\
&=&e_n^{1/2}\left\{E\left|E\left[\left.L(y_{n+1})-L(y_n)\right|\mathcal{F}_{t_n}\right]\right|^2\right\}^{1/2}.\label{S}
\end{eqnarray}
Submitting \eqref{S} into \eqref{4.31} and using Young's inequality $ab\leq\frac{1}{2}(a^2+b^2)$, we obtain
\begin{equation}\label{ditui}
e_{n+1}\leq e_n+E|L(y_{n+1})-L(y_n)|^2+he_n+h^{-1}\left\{E\left|E\left[\left.L(y_{n+1})-L(y_n)\right|\mathcal{F}_{t_n}\right]\right|^2\right\}.
\end{equation}
Note that $y_{n+1}=\widetilde{Y}_{t_n,y_n}(t_{n+1})$.  Utilizing \eqref{4.23} and \eqref{4.24} in Lemma \ref{L4.4}, we have

(1) If $q$ is odd, then
\begin{eqnarray}\label{4.36}
e_{n+1}&\leq& e_n(1+h)+Kh^{q+1}+h^{-1}Kh^{q+2}\nonumber\\
&\leq&e_n(1+h)+Kh^{q+1}.
\end{eqnarray}
It follows from Gronwall inequality (see \cite[Lemma  $1.6.$]{MilsteinBook}) that
$e_n\leq Kh^q$, i.e., $e_n^{1/2}\leq Kh^{\frac{q}{2}}$, $n=0,1,\ldots,N$.

(2) If $q$ is odd, we similarly have
\begin{equation}\label{4.37}
e_{n+1}\leq e_n(1+h)+Kh^{q}.
\end{equation}
Thus, Gronwall inequality leads to $e_n^{1/2}\leq Kh^{\frac{q-1}{2}}$, $n=0,1,\ldots,N$. This completes the proof.
\end{proof}
\begin{rem}
In fact, if $f$, $g_r,~r=1,\dots,D$, and $\nabla L(y)$ are polynomials with degree no larger than $q-1$, then the quadrature formulas in method \eqref{4.22} are exactly equal to the integrals in \eqref{method2}. In this case,  method \eqref{4.22}  exactly preserves the invariant $L(y)$. For general cases, Theorem \ref{tho4} implies that the mean square order of invariants conservation of MAVF methods using numerical integration only depends on the order of quadrature formulas. 
%the larger the order $q$ of used quadrature formula becomes, the better the invariant are preserved. Thus, we say the method \eqref{method3} is  `almost conservative'.
%As is seen in the proof of Theorem \ref{tho4}, it is not necessary to use commutative noises to deduce the conclusion, which is different from the convergence analysis in Theorem \ref{tho3}.  In addition, the conclusion implies that the larger the order $q$ of used quadrature formula is, the better the invariant are preserved. Thus, we say the method \eqref{method3} is  ``almost conservative".

\end{rem}

\section{Numerical experiments}\label{S5}
\subsection{MAVF methods}
In this section, we implement numerical experiments to verify our theoretical analyses. And we show the superiority of MAVF methods when applied to conservative SDEs. 
\subsubsection{Example 1: Kubo oscillator}
 Consider the following stochastic harmonic oscillator in \cite{MilsteinSmp}
\begin{equation}\label{E1}
  \left\{
  \begin{split}
    \ud X_1(t)&=-aX_2(t)-\sigma X_2(t)\circ\ud W(t),  \\
    \ud X_2(t)&=aX_1(t)+\sigma X_1(t)\circ\ud W(t),
  \end{split}
  \right.
\end{equation}
where $a$ and $\sigma$ are constants, and $W(t)$ is a one-dimensional Brownian motion. The quadratic function $I(x_1,x_2)=\frac{1}{2}(x_1^2+x_2^2)$ is the invariant of system (\ref{E1}).
In the numerical test, we take $a=\sigma=1$ and initial value $(X_1(0),X_2(0))=(1,0)$.

\begin{figure}%[H]
\centering
\includegraphics[width=1\textwidth]{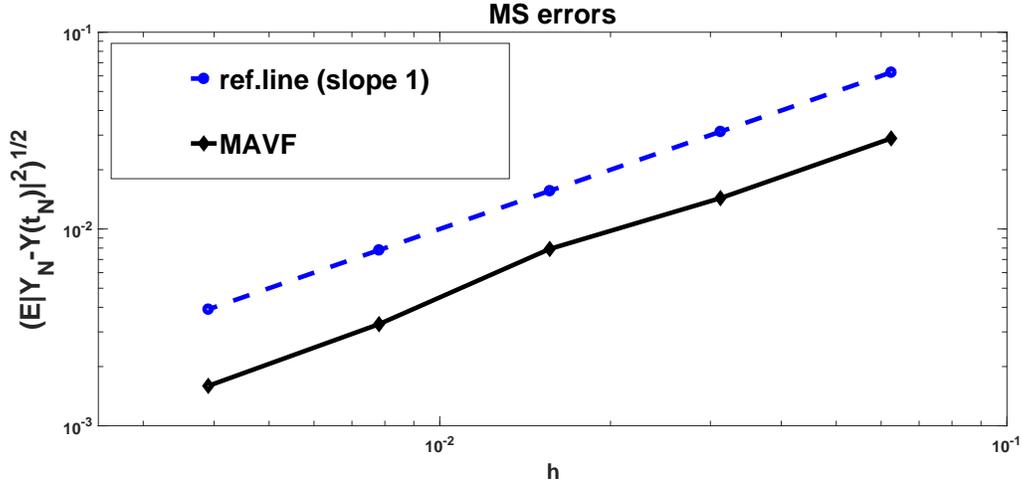}
\caption{Mean square errors of MAVF method at T=1 for Kubo oscillator. The dashed reference line has slope 1.}
\label{KO}
\end{figure}

\begin{figure}%[H]
\centering
\includegraphics[width=1\textwidth]{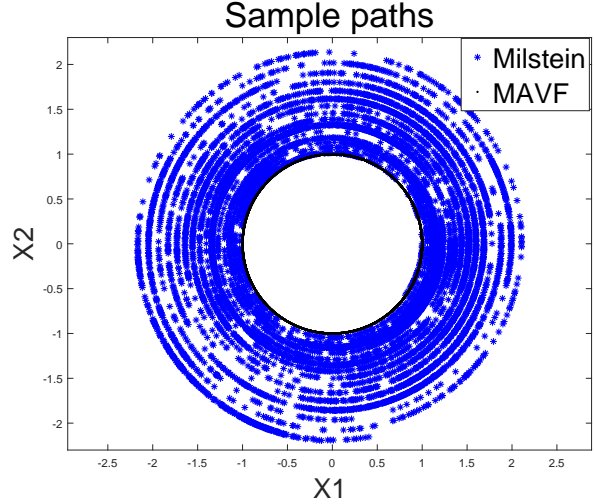}
\caption{Numerical sample paths of Milstein and  MAVF method for Kubo oscillator with $T=100$ and $h=0.01$.}
\label{KP}
\end{figure}

\begin{figure}%[H]
\centering
\includegraphics[width=1\textwidth]{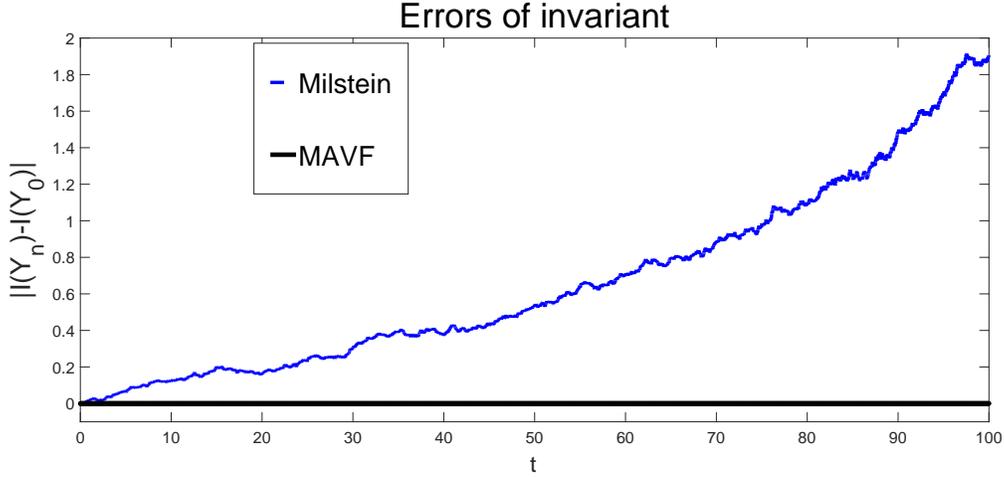}
\caption{Errors of invariant $I(x_1,x_2)=\frac{1}{2}(x_1^2+x_2^2)$ of Milstein and  MAVF method for Kubo oscillator with $T=100, h=0.01$.}
\label{KI}
\end{figure}

Figure \ref{KO} displays the convergence order of MAVF method (\ref{method1}). Here, the reference solution is obtained by Milstein method with stepsize $h_{ref}=2^{-14}$. The mean square errors are computed at the endpoint $T=1$ by adopting five different stepsizes $h=2^{-5},\,2^{-6},\,2^{-7},\,2^{-8},\,2^{-9}$. The expectation is realized by using the average of $1000$ independent sample paths. The convergence of order one, as is shown in this figure, is observed for the MAVF method, which is consistent with theoretical analyses of Theorem \ref{tho:1}.

Figure \ref{KP} and Figure \ref{KI} show the superiority of our MAVF method in aspect of numerically simulating the Kubo oscillator over a long time. Figure \ref{KP} displays the numerical solutions of Milstein method and MAVF method in the phase space along a single sample path. Here, $T=100$ and $h=0.01$. We observe that the numerical solutions of MAVF method remain on the unit circle, but the ones of Milstein method do not share this property. Figure \ref{KI} displays the errors of invariant $I(x_1,x_2)=\frac{1}{2}(x_1^2+x_2^2)$ , along one sample, using the Milstein method and MAVF method, respectively. This shows that MAVF method has a better long time stability.

\subsubsection{Example 2: Stochastic cyclic Lotka-Volterra system}
Consider the following stochastic dynamical system
\begin{equation}\label{E2}
  \ud\left(
  \begin{array}{c}
     X(t) \\
     Y(t)\\
     Z(t)
  \end{array}
  \right)=
   \left(
  \begin{array}{c}
     X(t)(Z(t)-Y(t)) \\
     Y(t)(X(t)-Z(t))\\
     Z(t)(Y(t)-X(t))
  \end{array}
  \right)\ud\,t+
     c\left(
  \begin{array}{c}
     X(t)(Z(t)-Y(t)) \\
     Y(t)(X(t)-Z(t))\\
     Z(t)(Y(t)-X(t))
  \end{array}
  \right)\circ\ud\,W(t),
\end{equation}
where $c$ is a real-valued constant and $W(t)$ is a one-dimensional Brownian motion. It can be regarded as a cyclic Lotka-Volterra system of competing 3-species in a chaotic environment \cite{Zwe}.
It is verified that system (\ref{E2}) has two conservative quantities
\begin{equation}
  I_1(x,y,z)=x+y+z,\qquad I_2(x,y,z)=x\cdot y\cdot z.
\end{equation}
In this experiment, we set $c=0.5$ and initial value $(X_0,Y_0,Z_0)=(1,2,1)$. Then, the exact solution of system (\ref{E2}) remains on the one-dimensional manifold
\begin{equation*}
  \mathcal M=\left\{(x,y,z)\in \mathbb R^3|I_1(x,y,z)=X_0+Y_0+Z_0,\quad I_2(x,y,z)=X_0\cdot Y_0\cdot Z_0\right\},
\end{equation*}
which is a closed curve in three-dimensional Euclid space. We compare the MAVF method with Milstein method to demonstrate the strengths of the proposed method.

\begin{figure}%[H]
\centering
\includegraphics[width=1\textwidth]{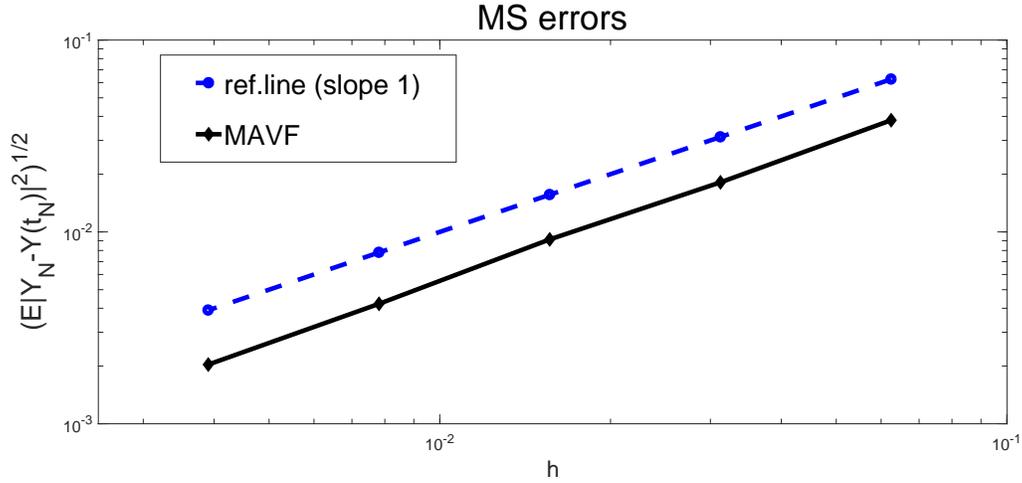}
\caption{Mean square errors of MAVF method at T=1 for stochastic cyclic Lotka-Volterra system. The dashed reference line has slope 1.}
\label{LO}
\end{figure}

\begin{figure}%[H]
\centering
\includegraphics[width=1\textwidth]{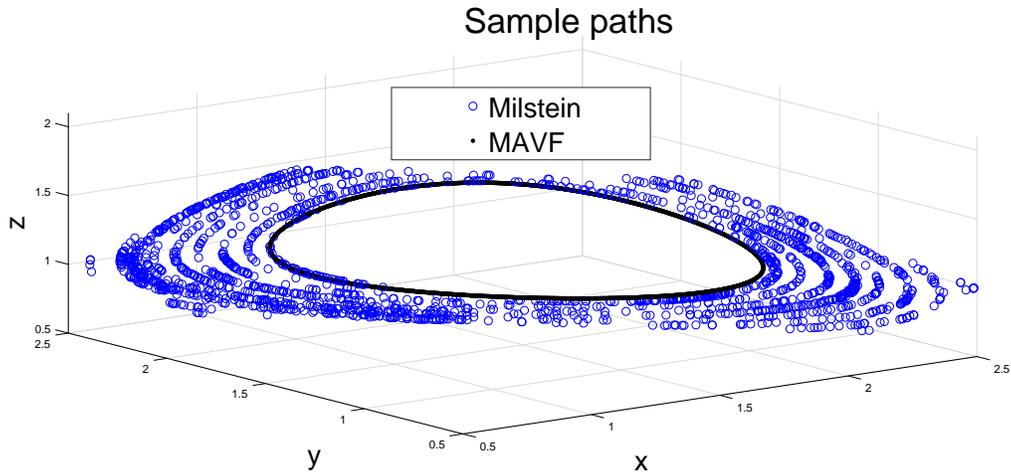}
\caption{Numerical sample paths of Milstein method and MAVF method for stochastic cyclic Lotka-Volterra system with $T=100$ and $h=0.01$.}
\label{LP}
\end{figure}

\begin{figure}%[H]
\centering
\includegraphics[width=1\textwidth]{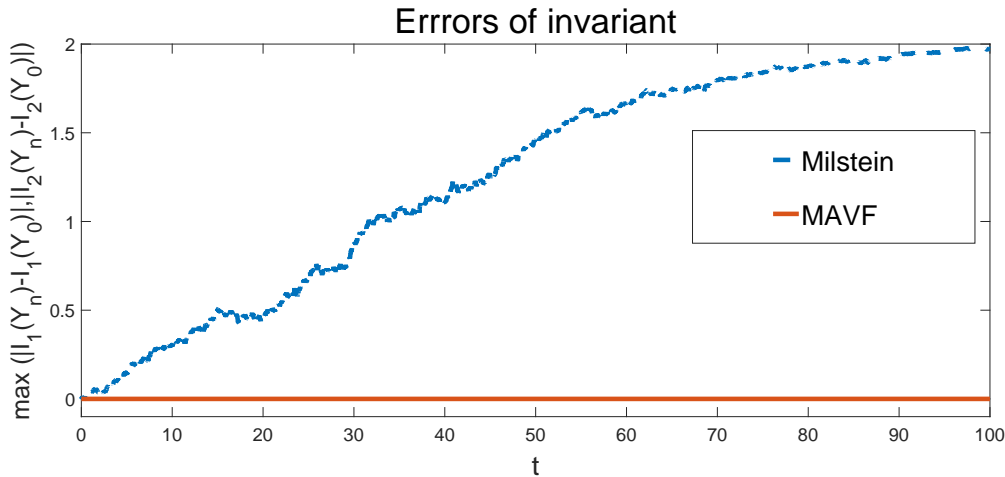}
\caption{Errors of invariants of Milstein method and MAVF method for stochastic cyclic Lotka-Volterra system with $T=100$ and $h=0.01$.}
\label{LI}
\end{figure}

Figure \ref{LO} shows the convergence order of MAVF method. The reference solution is obtained by Milstein method with step size $h_{ref}=2^{-14}$. The mean square errors are computed at the endpoint $T=1$ by adopting five different stepsizes $h=2^{-5},\,2^{-6},\,2^{-7},\,2^{-8},\,2^{-9}$. The expectation is approximated using the average of $1000$ independent sample paths. It is observed that the conservative method for this system is of mean square order $1$.

The numerical sample paths of Milstein method and MAVF method are shown in Figure \ref{LP}. The interval length $T=100$ and step size $h=0.01$. We observe that numerical solutions of the MAVF method, along one sample, lie in the manifold $\mathcal M$, but those of Milstein method do not. Figure \ref{LI} displays the errors of invariants of these two methods. Here the error is denoted by $\max\{|I_1(Y_n)-I_1(Y_0)|,|I_2(Y_n)-I_2(Y_0)|\}$.
%It is not hard to verify that the Milstein method for system (\ref{E2}) is able to exactly preserve the invariant $I_1$. Thus, in fact, the errors of invariants of Milstein method come from $I_2$. 
On the other hand, the MAVF method exactly preserves the two invariants, as is seen in this figure. Although the coefficients of system (\ref{E2}) do not satisfy the globally Lipschitz conditions as required in Theorem \ref{tho:2}, the MAVF method for original system still works well, which indicates that MAVF methods can be applied to more general system.

\subsubsection{Example 3: Stochastic Hamiltonian system with multiple invariants}
In this experiment, we consider the following Stochastic Hamiltonian system with commutative noises
\begin{eqnarray}\label{E3}
\ud\left(\begin{array}{c}
           Y_1(t) \\
           Y_2(t) \\
           Y_3(t) \\
           Y_4(t)
         \end{array}
\right)=
\left(\begin{array}{c}
           Y_3(t) \\
           Y_4(t) \\
           -Y_1(t) \\
           -Y_2(t)
         \end{array}
\right)\left(\ud t+c_1\circ\ud W_1(t)+c_2\circ\ud W_2(t)\right),
\end{eqnarray}
where $c_1$ and $c_2$ are constants, and $W_1(t)$ and $W_2(t)$ are two independent Brownian motions. The system (\ref{E3}) can be regarded as the extension of Example $3.1$ in \cite{Misasym}. One can verify that this system has three invariants
\begin{equation}\label{L123}
  \begin{split}
      L_1(y_1,y_2,y_3,y_4)&= y_1y_4-y_2y_3, \\
      L_2(y_1,y_2,y_3,y_4)&= \frac{1}{2}(y_1^2-y_2^2+y_3^2-y_4^2),\\
      L_3(y_1,y_2,y_3,y_4)&=y_1y_2+y_3y_4.
  \end{split}
\end{equation}
In this experiment, we take parameters $c_1=1,\,c_2=0.5$ and initial value $Y_0=(-0.5,0,0.5,1)^\top$.
We still compare MAVF method with Milstein method for system (\ref{E3}).

We can observe from Figure \ref{HO} that the mean quare convergence of order one when applying the MAVF method to system (\ref{E3}). The mean square errors are computed at the endpoint $T=1$ by adopting five different stepsizes $h=2^{-5},\,2^{-6},\,2^{-7},\,2^{-8},\,2^{-9}$. The reference solution is obtained by Milstein method with step size $h_{ref}=2^{-14}$. The expectation is evaluated by average of $1000$ independent sample paths. This verifies the conclusion about convergence in Theorem \ref{tho:2} under the case of commutative noises.

\begin{figure}
\centering
\includegraphics[width=1\textwidth]{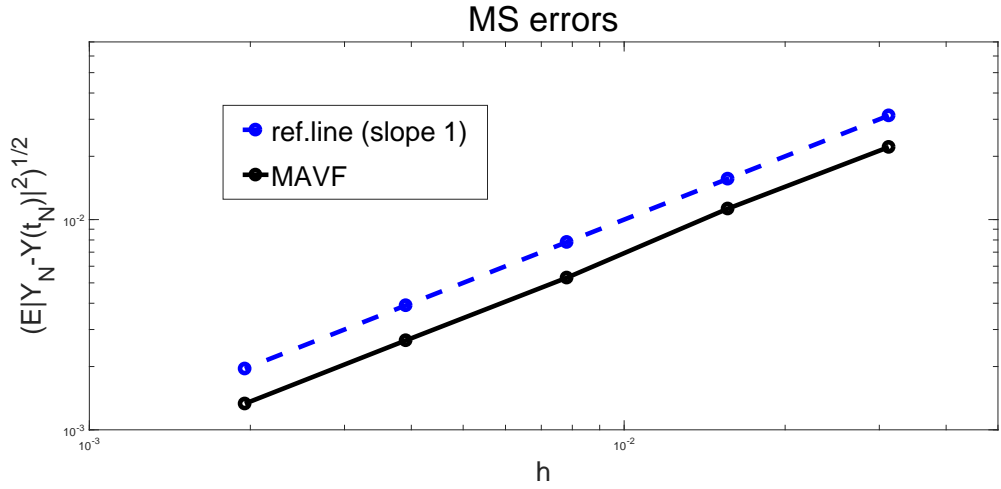}
\caption{Mean square errors of MAVF method at T=1 for stochastic Hamiltonian system. The dashed reference line has slope 1.}
\label{HO}
\end{figure}

\begin{figure}%[H]
  \centering
  \subfloat[]{\includegraphics[width=0.5\textwidth]{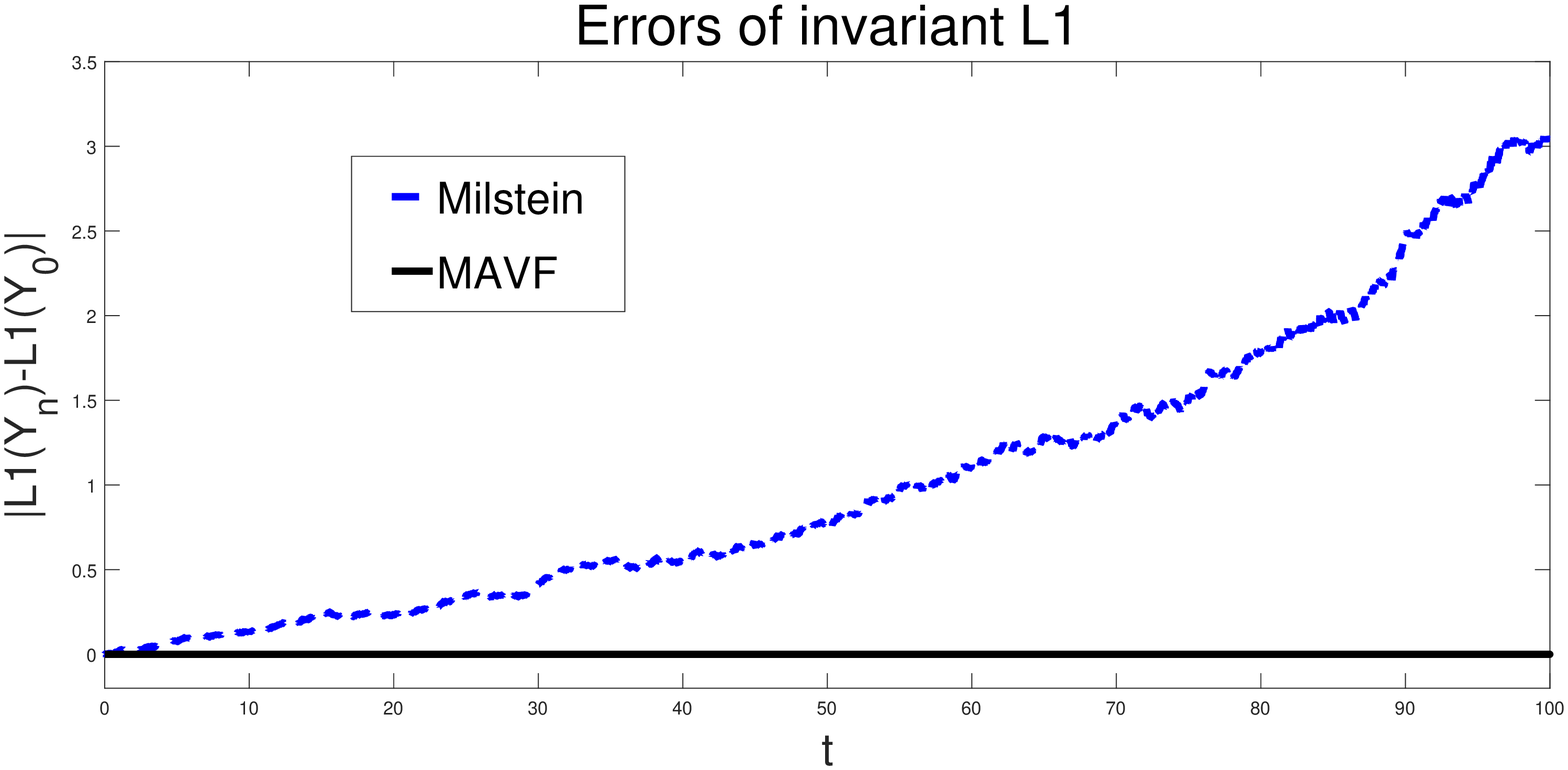}}
    \\
    \subfloat[]{\includegraphics[width=0.5\textwidth]{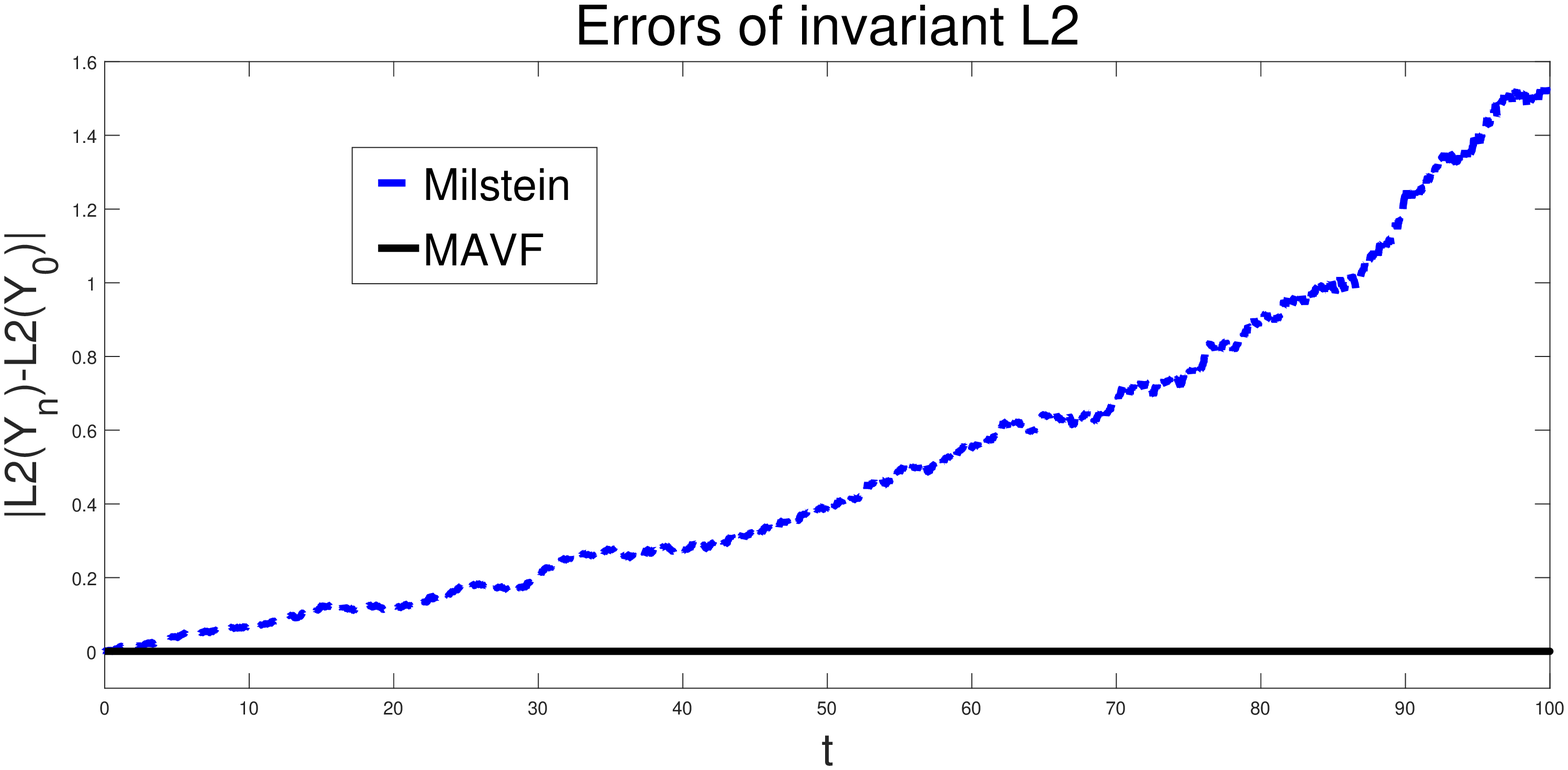}}
    \subfloat[]{\includegraphics[width=0.5\textwidth]{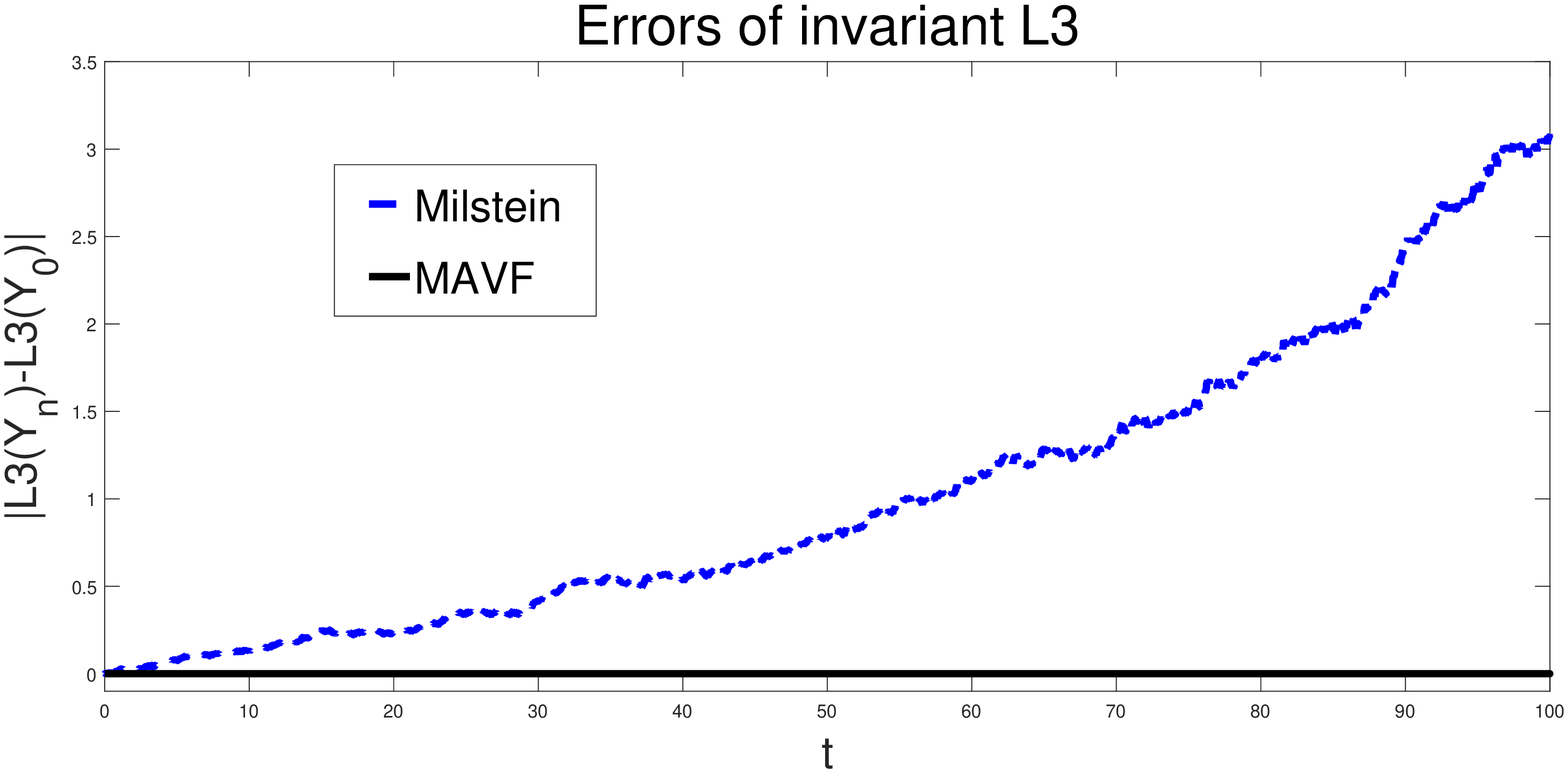}} \\
  \caption{Errors of invariants of Milstein method and MAVF method for stochastic Hamiltonian system with $T=100$ and $h=0.01$.}
    \label{HI}
    \vspace{0.2in}
\end{figure}

As in previous experiments, Figure \ref{HI} displays the errors of invariants $L_1,\,L_2$ and $L_3$, respectively, when applying MAVF method and Milstein method. Here we set $T=100$ and $h=0.01$. We  observe that  MAVF method for system (\ref{E3}) preserves exactly these three invariants but Milstein method fails. This shows that MAVF method possesses a better long time stability.

\subsection{MAVF methods using numerical integration}
In this section, we perform numerical experiments to present the effect of numerical integration on MAVF methods. The MAVF methods \eqref{method3} using quadrature formula \eqref{4.2}, \eqref{4.3}, \eqref{4.4} and \eqref{Q6} are called MAVF-Q2 method, MAVF-Q3 method, MAVF-Q4 method and MAVF-Q6 method respectively.

Consider the following SDE with commutative noises (see \cite{Ccc})
\begin{small}
	\begin{equation}\label{E4}
	\quad\ud\left(
	\begin{array}{c}
	p \\
	q
	\end{array}
	\right)=\left(
	\begin{array}{cc}
	0 & -1 \\
	1 & 0
	\end{array}\right)
	\left(\begin{array}{c}
	p \\
	\sin(q)
	\end{array}
	\right)\ud t
	+
	\left(
	\begin{array}{cc}
	0 & -\cos(q)\\
	\cos(q) & 0
	\end{array}\right)
	\left(\begin{array}{c}
	p \\
	\sin(q)
	\end{array}
	\right)
	\left(c_1\circ\ud W_1(t)+c_2\circ\ud W_2(t)\right),
	\end{equation}
\end{small}
where $c_1$, $c_2$ are constants, and $W_1(t)$, $W_2(t)$ are two independent Brownian motions. This system has $I(p,q)=\frac{1}{2}p^2-\cos(q)$ as its invariant. We take $c_1=1$, $c_2=0.5$, and initial value $(p_0,q_0)=(0.2,1)$ in this experiment.

Figure \ref{PO} shows the convergence order of MAVF-Q2 method. The reference solution is obtained by Milstein method with step size $h_{ref}=2^{-14}$. The mean square errors are computed at the endpoint $T=1$ by adopting five different stepsizes $h=2^{-5},\,2^{-6},\,2^{-7},\,2^{-8},\,2^{-9}$. The expectation is approximated using the average of $1000$ independent sample paths. It is observed that the conservative method for this system is of mean square order $1$, which is consistent with the conclusion of Theorem \ref{tho3}.

Figure \ref{PP} presents the sample paths of MAVF method, MAVF-Q2 method, MAVF-Q4 method and MAVF-Q6 method. The computational interval is T=10000 and stepsize h=0.01. Table \ref{tab} shows the errors of invariant of these three methods along single sample path and their computation times. As is seen in Figure \ref{PP} and Table \ref{tab}, as the order of quadrature formula enlarges, the invariant is preserved better. 

Figure \ref{PIO} shows mean square orders of invariant conservation of MAVF methods using numerical integration. Here, we use MAVF-Q2 method, MAVF-Q3 method, MAVF-Q4 method to perform numerical experiment.  The reference solution is obtained by Milstein method with step size $h_{ref}=2^{-14}$. The mean square orders of invariant conservation are computed at the endpoint $T=1$ by adopting five different stepsizes $h=2^{-6},\,2^{-7},\,2^{-8},\,2^{-9},\,2^{-10}$. The expectation is approximated using the average of $1000$ independent sample paths. It is shown that MAVF-Q2 method and MAVF-Q3 method have mean square order $1$ of invariant conservation, while MAVF-Q4 method has mean square order $2$ of invariant conservation. These results coincide with those of Theorem \ref{tho4}.   
\begin{figure}%[H]
	\centering
	\includegraphics[width=1\textwidth]{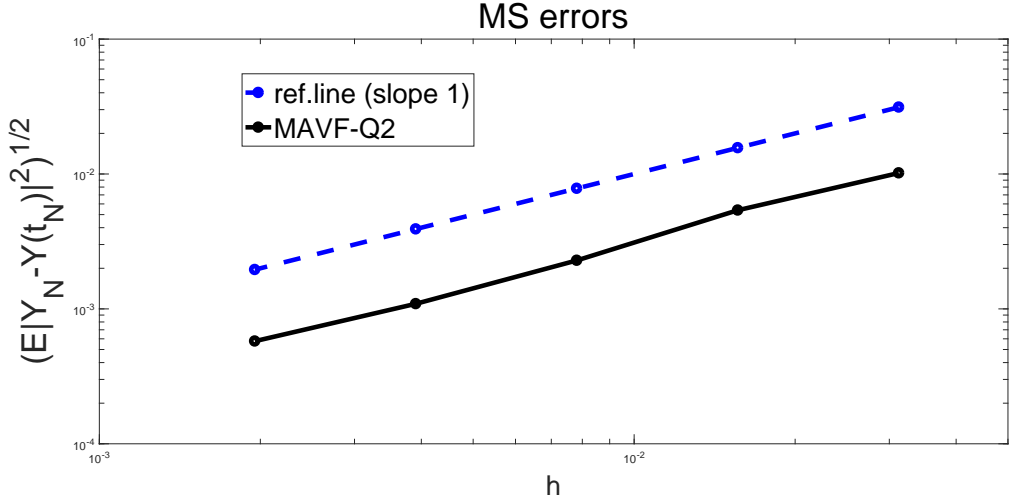}
	\caption{Mean square errors of MAVF-Q2 method at T=1 for stochastic pendulum problem. The dashed reference line has slope 1.}
	\label{PO}
\end{figure}

\begin{figure}%[H]
	\centering
	\subfloat[]{\includegraphics[width=0.5\columnwidth]{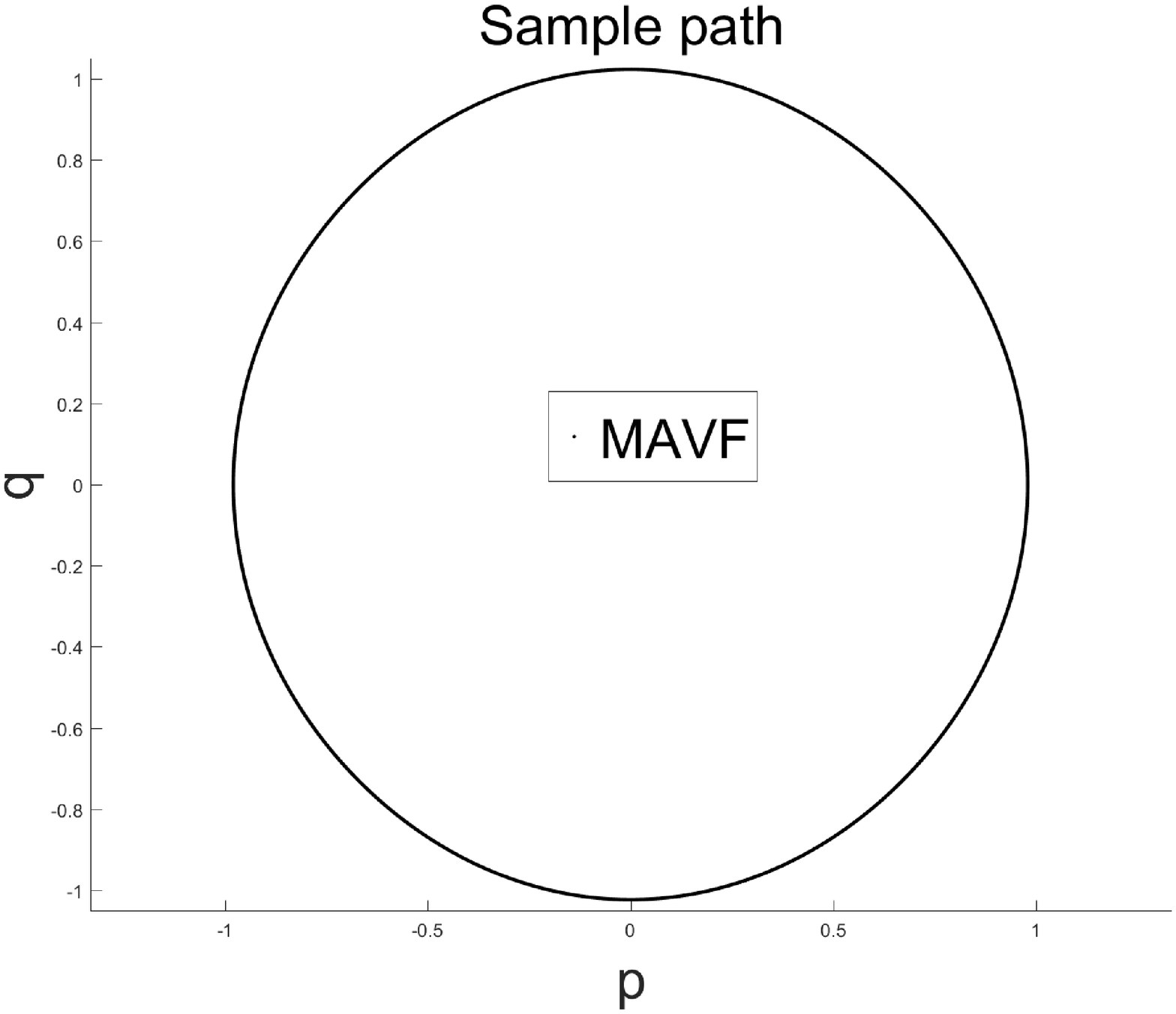}}
	\subfloat[]{\includegraphics[width=0.5\columnwidth]{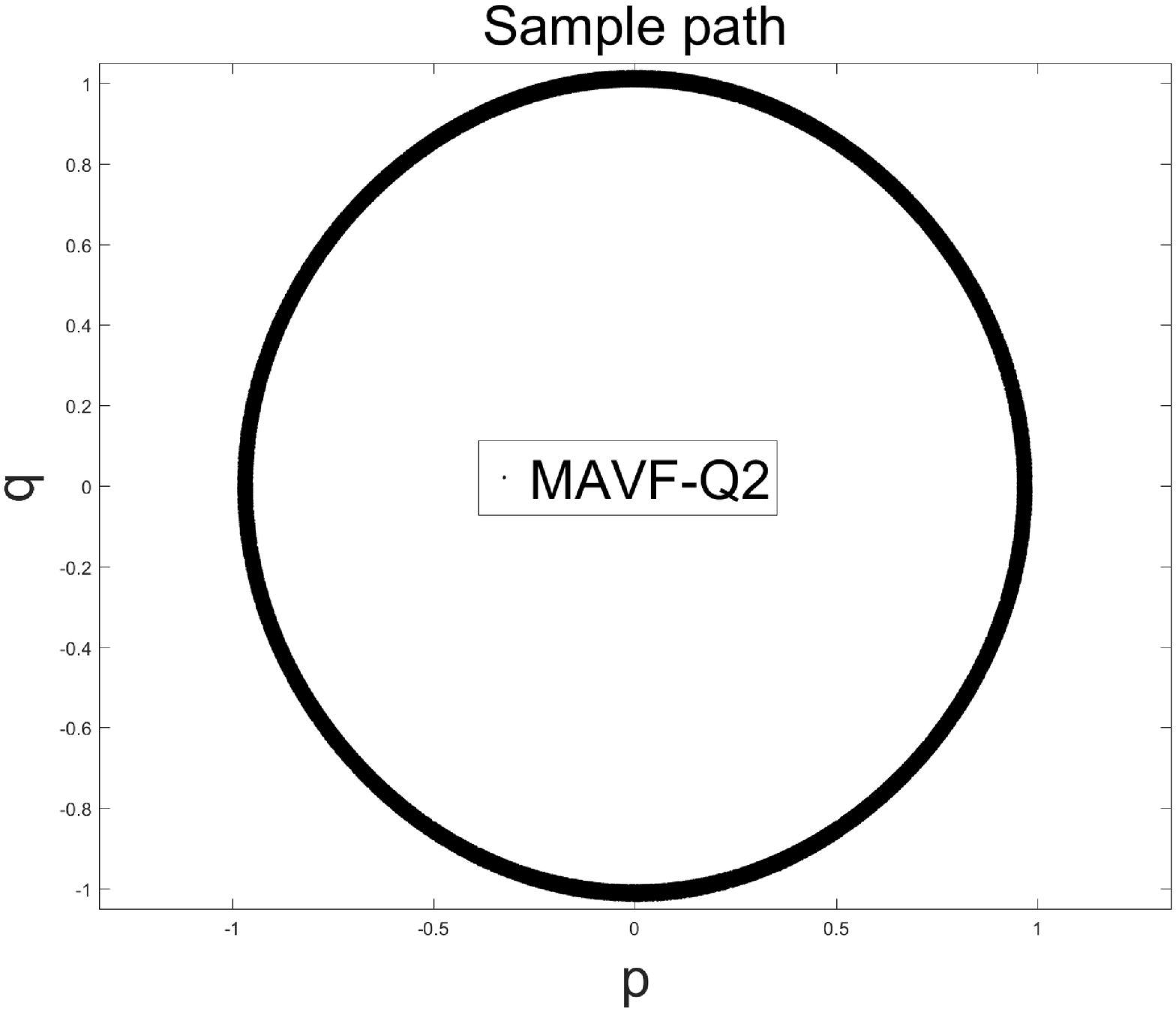}}\\
	\subfloat[]{\includegraphics[width=0.5\columnwidth]{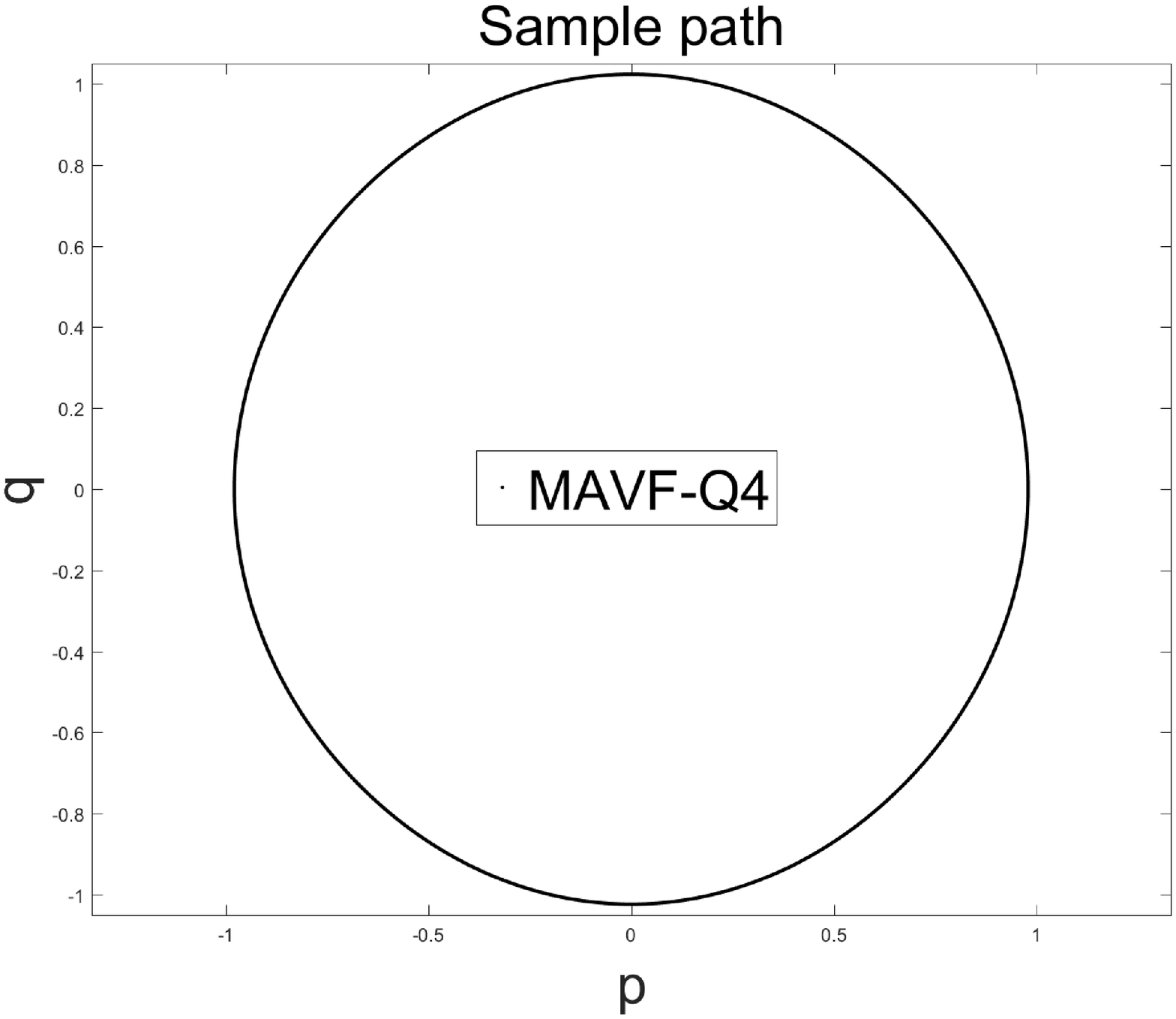}}
	\subfloat[]{\includegraphics[width=0.5\columnwidth]{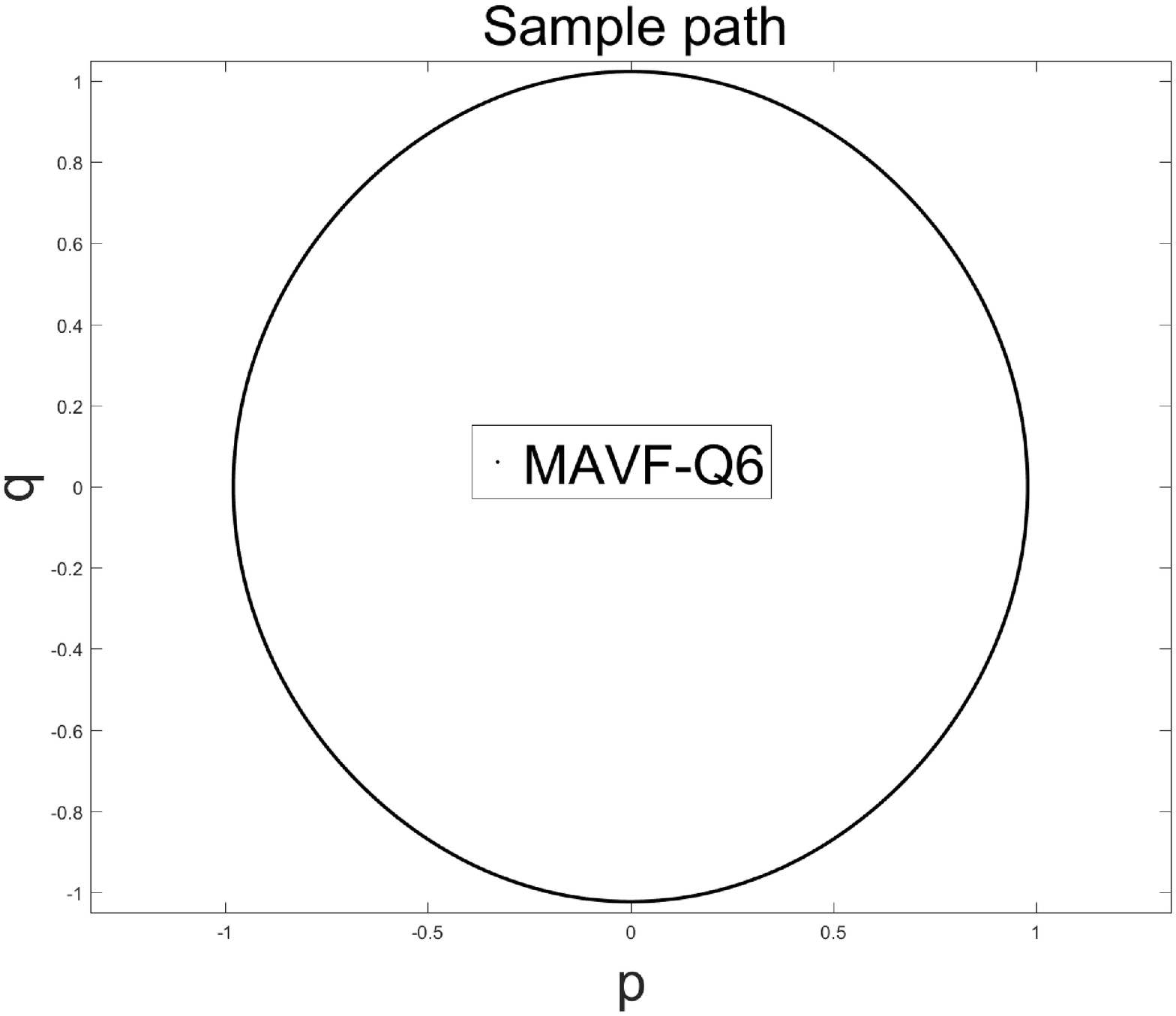}}
	\caption{Numerical sample paths of MAVF method, MAVF-Q2 method, MAVF-Q4 method and MAVF-Q6 method for stochastic pendulum problem with $T=10000$ and $h=0.01$.}
	\label{PP}
	\vspace{0.2in}
\end{figure}

\begin{table}%[h]
	\centering
	\caption{Errors of invariant of MAVF-Q2 method, MAVF-Q4 method, and MAVF-Q6 method at different times along single sample path for stochastic pendulum problem with T=10000 and h=0.01.}\label{tab}
	\begin{tabular}{ccccccc}
		\toprule
		t& 100& 500& 1000& 5000& 10000&CPU time\\
		\midrule
		MAVF-Q2&0.0009&0.0036&0.0189&0.0193&0.0101&78s   \\
		MAVF-Q4&0.0145E-05& 0.2108E-05&  0.4764E-05& 0.8180E-05&  0.9179E-05&61s\\
		MAVF-Q6&0.0007E-05&    0.0095E-05&  0.0170E-05 &   0.0738E-05 &   0.1478E-05&59s\\
		\bottomrule
	\end{tabular}
	%\caption{Errors of invariant of MAVF-Q2 method, MAVF-Q4 method, and MAVF-Q6 method at different times along single sample path}\label{tab}
\end{table}

\begin{figure}%[H]
	\centering
	\includegraphics[width=1\textwidth]{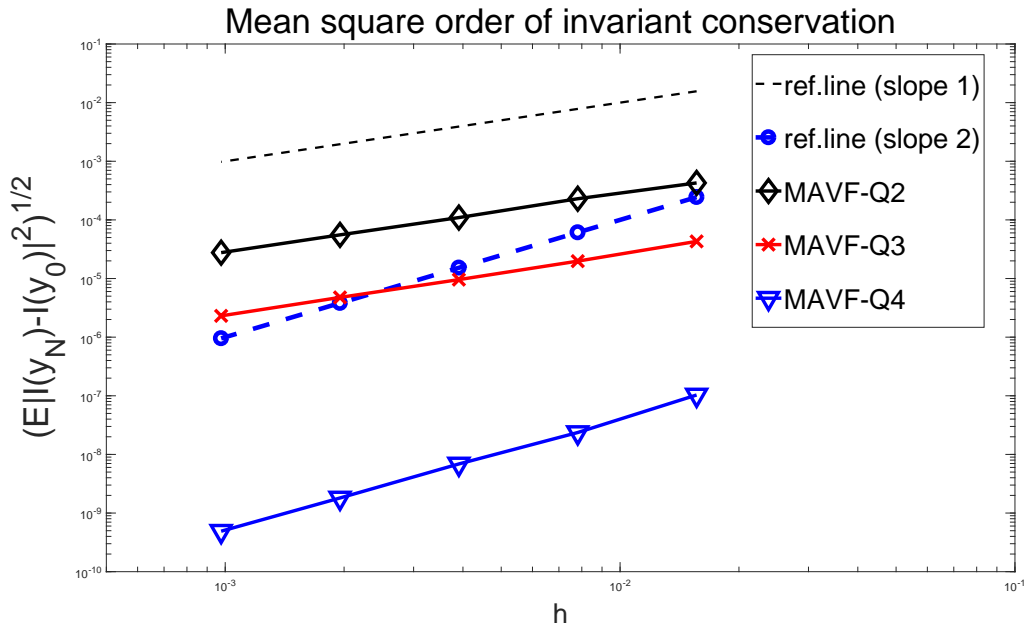}
	\caption{Mean square orders of invariant conservation of MAVF-Q2 method, MAVF-Q3 method and MAVF-Q4 method at T=1 for stochastic pendulum problem. The two dashed reference lines have slope 1 and 2 respectively.}
	\label{PIO}
\end{figure}

%\section{Conclusion}\label{S5}
%In this paper, we construct the MAVF methods for conservative stochastic systems based on LIMs \cite{Brugbook} and AVF methods \cite{Ccc}. The MAVF methods exactly preserve multiple invariants of conservative SDEs with single or multiple noises. We prove that the MAVF methods are of mean square order of $1$ in case of single noise or commutative noises. Finally, numerical experiments are performed to verify our theoretical results. Some problems remain to be resolved: how to construct high strong order of MAVF method, and how numerical quadrature formulas influence the convergence order of MAVF methods. These problems will be studied in our further work.
%\section*{Acknowledgments}
%The authors would like to thank the anonymous referees for their valuable suggestions.

%\section*{References}
\bibliographystyle{plain}
\bibliography{mybibfile}

\end{document}